\documentclass[11pt]{amsart}
\usepackage{fullpage, float}
\usepackage{amssymb,amsmath,amscd,mathrsfs,graphicx, enumerate}
\usepackage{tikz-cd, tikz, verbatim}
\usepackage{caption, subcaption}

\theoremstyle{plain}
\newtheorem{theorem}{Theorem}[section]
\newtheorem{lemma}{Lemma}[section]
\newtheorem{prop}{Proposition}[section]

\newtheorem*{claim*}{Claim}

\theoremstyle{definition}
\newtheorem{definition}{Definition}[section]
\newtheorem{remark}{Remark}[section]

\newcommand{\PP}{\mathbb{P}}
\newcommand{\CC}{\mathbb{C}}
\newcommand{\QQ}{\mathbb{Q}}
\newcommand{\ZZ}{\mathbb{Z}}
\newcommand{\calO}{\mathcal{O}}
\newcommand{\calN}{\mathcal{N}}
\newcommand{\calB}{\mathcal{B}}
\newcommand{\calR}{\mathcal{R}}

\newcommand{\zz}{\mathbb{Z}/2\mathbb{Z}}
\newcommand{\PGL}{\mathbb{P}\mathrm{GL}}

\newcommand{\GL}{\mathrm{GL}}
\newcommand{\diag}{\mathrm{diag}}
\newcommand{\Bl}{\mathrm{Bl}}

\newcommand{\rk}{\mathrm{rk}}
\newcommand{\Stab}{\mathrm{Stab}}
\newcommand{\Aut}{\mathrm{Aut}}
\newcommand{\codim}{\mathrm{codim}}

\title[Cohomology of the moduli space of degree two Enriques surfaces]{Cohomology of the moduli space of degree two Enriques surfaces}
\author[M. Fortuna]{Mauro Fortuna}
\address{Institut für Algebraische Geometrie, Leibniz Universität Hannover, Welfengarten 1, 30167 Hannover, Germany.}
\email{fortuna@math.uni-hannover.de}

\begin{document}
\maketitle

\begin{abstract}
    We compute the intersection Betti numbers of the GIT model of the moduli space of numerically polarized Enriques surfaces of degree 2. The strategy of the cohomological calculation relies on a general method developed by Kirwan to compute the cohomology of GIT quotients of projective varieties, based on the equivariantly perfect stratification of the unstable points studied by Hesselink and others and a partial resolution of singularities, called Kirwan blow-up.
\end{abstract}

\section{Introduction}	

Moduli spaces of (numerically) polarized Enriques surfaces and their geometrically meaningful compactifications are a topic of increasing interest in algebraic geometry (cf. \cite{GH16}, \cite{CDGK18}, \cite{K20} and \cite[Chapter 5]{CDL20}). One of the most interesting aspects one wants to understand is the topology of these spaces. From that perspective, the present article provides the first result about the cohomology of a moduli space of (numerically) polarized Enriques surfaces. More precisely, the purpose of this paper is to compute the intersection Betti numbers of the moduli space of numerically polarized Enriques surfaces of degree 2. 

The projective model of degree 2 Enriques surfaces was firstly constructed by Horikawa in \cite{Hor781}. The K3 coverings of these Enriques surfaces are hyperelliptic quartic K3 surfaces given as double coverings of $\PP^1 \times \PP^1$ branched over a curve of bidegree $(4,4)$ invariant under a suitable involution $\iota: \PP^1 \times \PP^1 \rightarrow \PP^1 \times \PP^1 $ with four fixed points. This transformation induces the Enriques involution on the K3 surface. Therefore, by looking at the isomorphism classes of such branch curves on $\PP^1\times \PP^1$, we can construct the GIT quotient
$$ M^{GIT}:= \PP H^0(\PP^1\times \PP^1, \calO_{\PP^1\times \PP^1}(4, 4))^{\iota}/\!\!/(\CC^*)^2 \rtimes D_8, $$
where $\PP H^0(\PP^1\times \PP^1, \calO_{\PP^1\times \PP^1}(4, 4))^{\iota}$ is the linear subsystem of $|\calO_{\PP^1\times \PP^1}(4, 4)|$ of $\iota$-invariant curves and $(\CC^*)^2 \rtimes D_8$ is the subgroup of the automorphisms of $\PP^1 \times \PP^1$ that commute with $\iota$. The quotient $M^{GIT} $ can thus be seen as a compactification of the moduli space of numerically polarized Enriques surfaces of degree 2 (cf. Theorem \ref{thm:polarized}). The purpose of the present article is to compute the intersection cohomology of $M^{GIT}$. 

The space $M^{GIT} $ was extensively studied by Horikawa (\cite{Hor781} and \cite{Hor782}), Shah (\cite{Sha81}) and Sterk (\cite{Ste91} and \cite{Ste95}). More precisely, Horikawa proved the Torelli Theorem for Enriques surfaces and studied the period map (and its extension) from $M^{GIT}$ to the period domain of Enriques surfaces. Shah classified all the projective degenerations of Horikawa's model of Enriques surfaces. Later Sterk built on these results by dealing with compactifications of the period space of Enriques surfaces which are of geometric interest. In particular, he gave a description of the boundary in the Baily-Borel compactification of the period space in \cite{Ste91} and constructed a resolution of the period map via a new geometrically meaningful compactification, called Shah compactification. This space can be obtained as a double weighted blow-up of $M^{GIT}$ and its points include all the degenerations of Enriques surfaces classified in \cite{Sha81}. Moreover, in \cite{Ste95} the resolution of the period map was proved to factorize through a semi-toric compactification, obtained as normalized blow-up of the Baily-Borel compactification
along the closure of the divisor describing periods of Enriques surfaces with a `special' quasi-polarization (see Section \S \ref{sec:Horikawa} for the definition of `special' quasi-polarization). It seems interesting to study the cohomological differences between all the aforementioned compactifications. 

The strategy to compute the intersection Betti numbers of $ M^{GIT} $ relies on a general procedure developed by Kirwan to calculate the cohomology of GIT quotients (see \cite{Kir84}, \cite{Kir85}, \cite{Kir86}). The crucial step of that method consists of the construction of a partial desingularization $ M^K \rightarrow M^{GIT}$, known as \textit{Kirwan blow-up}, having only finite quotient singularities, obtained by successively blowing up the loci parametrizing strictly polystable points in the parameter space. Then one is to compute the Hilbert-Poincar\'{e} polynomial of $M^K$ and descend back to the GIT quotient $M^{GIT}$ using the \textit{Decomposition Theorem}. 

Examples of application of Kirwan's method are the topological descriptions of the moduli space of points on the projective line (\cite[\S 8]{MFK94}), of $ K3 $ surfaces of degree 2 (\cite{Kir88}) and of hypersurfaces in $ \PP^n $ (\cite{Kir89}), with explicit complete computations only in the case of plane curves up to degree $ 6 $, cubic and quartic surfaces. More recently, the procedure has been applied to compactifications of the moduli space of cubic threefolds in \cite{CMGHL19} and that of non-hyperelliptic curves of genus four in \cite{F18} .

Our result is summarised by the following:

	\begin{theorem}\label{thm:main}
		The intersection Betti numbers of $M^{GIT} $ and the Betti numbers of the Kirwan blow-up $ M^K $ are as follows:
		\begin{center}
			\begin{tabular}{r|ccccccccccc}
				$ i $&0&2&4&6&8&10&12&14&16&18&20\\ \hline \rule{0pt}{2.5ex}
				$\dim IH^i(M^{GIT}, \QQ)$ & 1& 1 & 2 & 2 & 3 & 3 & 3 & 2 & 2 & 1 & 1\\[1ex]
				$\dim H^i(M^K, \QQ)$& 1 & 4 & 8 & 13 & 18 & 20 & 18 & 13 & 8 & 4 & 1\\
			\end{tabular}
		\end{center}
		while all the odd degree (intersection) Betti numbers vanish.
	\end{theorem}

The structure of the paper reflects the steps of Kirwan's machinery. Section \S \ref{sec:Horikawa} is devoted to the description of Horikawa's model, which gives rise to Enriques surfaces with a non-special polarization of degree 2. In Section \S \ref{sec:background} we use this model to construct the moduli space $M^{GIT}$ as GIT quotient $X/\!\!/G$, which can be seen as a compactification of the moduli space of numerically polarized Enriques surfaces of degree 2 (cf. Theorem \ref{thm:polarized}). Moreover, the geometrical description of the semistable and stable loci are presented. In Section \S \ref{sec:strati}, we calculate the equivariant Hilbert-Poincar\'{e} polynomial of the semistable locus $X^{ss}$ in the parameter space of $(4,4)$ $\iota$-invariant curves (see Proposition \ref{prop:seriesss}). This is done by computing the \textit{Hesselink-Kempf-Kirwan-Ness (HKKN) stratification} of the unstable locus, naturally associated to the linear action of $ G $ on the parameter space $ X $, followed by an excision type argument. In Section \S \ref{sec:blow}, we explicitly construct the partial desingularization $ M^K \rightarrow M^{GIT} $, by blowing up three $ G $-invariant loci in the GIT boundary of $ M^{GIT} $, corresponding to strictly polystable curves (cf. Proposition \ref{prop:polystable}). Section \S \ref{sec:cohomology} is devoted to the computation of the rational Betti numbers of the Kirwan blow-up $ M^K $ (see Theorem \ref{thm:cohoblow}). Here the correction terms arising from the modification process $M^K \rightarrow M^{GIT}$ are divided into a main and an extra contribution: the former takes into account the geometry of the centres of the blow-ups and the latter the action of $G$ on the exceptional divisors. In the end, the intersection Betti numbers of $ M^{GIT} $ are computed in Section \S \ref{sec:intersection}, as an application of the \textit{Decomposition Theorem} (cf. \cite{BBD82}) to the blow-down operations at the level of parameter spaces (see Theorem \ref{thm:intM}).

\subsection*{Notation and conventions} We work over the field of complex numbers and all the cohomology and homology theories are taken with \textit{rational} coefficients. The intersection cohomology will be always considered with respect to the middle perversity (see \cite{KW06} for an excellent introduction). For any topological group $ G $, we will denote by $ G^0 $ the connected component of the identity in $ G $ and by $ \pi_0(G):=G/G^0 $ the finite group of connected components of $ G $. The universal classifying bundle of $ G $ will be denoted by $ EG\rightarrow BG $. If $ G $ acts on a topological space $ Y $, its equivariant cohomology (see \cite{AB83}) will be defined to be $ H^*_G(Y):=H^*(Y\times_G EG) $. The Hilbert-Poincar\'{e} series is denoted by 
$$ P_t(Y):=\sum_{i\geq 0}t^i \dim H^i(Y), $$
and analogously for the intersection and equivariant cohomological theories. If $ F $ is a finite group acting on a vector space $ A $, then $ A^F $ will indicate the subspace of elements in $ A $ fixed by $ F $.

\subsection*{Acknowledgements} I wish to thank my PhD advisor Klaus Hulek for his valuable suggestions and useful comments, and Giacomo Mezzedimi for so many helpful discussions. This work is partially supported by the DFG Grant Hu 337/7-1.

\section{Horikawa's model}\label{sec:Horikawa}
An Enriques surface is a smooth compact complex surface $S$ such that $H^1(\calO_S)=0$ and its canonical bundle $\omega_S$ is not trivial, but $\omega_S^{\otimes 2}\cong \calO_S$. The last condition implies the existence of an \'{e}tale double covering $T\rightarrow S$ and by surface classification $T$ is a K3 surface. Moreover, by definition every Enriques surface is algebraic, in particular $\mathrm{NS}(S)\cong H^2(S, \ZZ)$. The canonical class is the only torsion element in the N\'{e}ron-Severi group and there is a non-canonical splitting $H^2(S, \ZZ)=H^2(S, \ZZ)_f \oplus \zz$ where $H^2(S, \ZZ)_f=H^2(S, \ZZ)/\mathrm{torsion}$ is a free module of rank 10. The intersection product endows this with a lattice structure and 
$$ H^2(S, \ZZ)_f = \mathrm{Num}(S) \cong U\oplus E_8(-1), $$
where $U$ denotes the hyperbolic plane and $E_8(-1)$ is the only negative definite, even, unimodular lattice of rank 8. 

A polarized (resp. numerically polarized) Enriques surface is a pair $(S, H)$, where $S$ is an Enriques surface and $H\in \mathrm{NS}(S)$ (resp. $H\in \mathrm{Num}(S)$) is the (numerical) class of an ample line bundle. Moreover, a quasi-polarization is a nef and big line bundle, not necessarily ample. The degree of a (numerical) (quasi-)polarization is its self-intersection and it is always even by adjunction. 

In the present article we consider only quasi-polarizations of degree 2. By \cite[Remark 5.7.10]{CDL20}, each numerical quasi-polarization of degree 2 can be represented as a sum of two isotropic classes $f+g$ in $U\oplus E_8(-1)$, with one of the following properties: 
\begin{enumerate}[(i)]
    \item Both $f$ and $g$ are nef: in this case $f+g$ is ample and is called \textit{non-special} polarization;
    \item The class $f-g$ represents an effective divisor $R$ with $R^2=-2$, $Rf=-1$, and hence $f=g+R$ is not nef: in this case $f+g$ is not ample and is called \textit{special} quasi-polarization.
\end{enumerate}

We now present a geometrical construction of Enriques surfaces together with a numerical polarization of degree 2, firstly given by Horikawa in \cite{Hor781} (cf. also \cite[V.23]{BHPV04}). Let $\PP^1 \times \PP^1$ be acted on by the involution:
$$ \iota:\PP^1\times \PP^1 \rightarrow \PP^1 \times \PP^1 $$
$$ (x_0:x_1, y_0: y_1)\mapsto (x_0:-x_1, y_0:-y_1). $$
It has four isolated fixed points, namely 
$$ \Delta:=\{ (0:1, 0:1), \ (0:1, 1:0), \ (1:0, 1:0), \ (1:0, 0:1) \}. $$
Let $B$ be a reduced curve on $\PP^1 \times \PP^1$ of bidegree $(4,4)$ which is invariant under $\iota$, does not pass through any point of $\Delta$ and has at worst simple singularities. The minimal resolution of the double covering of $\PP^1 \times \PP^1$ branched over $B$ is a K3 surface $T\rightarrow \PP^1 \times \PP^1$. The pullback of the $(1,1)$-class on $\PP^1\times \PP^1$ endows $T$ with a polarization of degree 4, which splits as a sum of two genus one fibrations corresponding to the pullbacks of the two rulings on $\PP^1 \times \PP^1$. Moreover, the involution $\iota$ on $\PP^1 \times \PP^1$, composed with the deck transformation of the double covering, induces a fixed point free involution $\sigma$ on $T$. Therefore the quotient $T \rightarrow S:= T/\langle \sigma \rangle$ is an Enriques surface. As the degree 4 polarization on $T$ is invariant under $\sigma$, it induces a polarization $L$ of degree 2 on the Enriques surface $S$. This ample line bundle on $S$ splits as a sum $L=E+F$ of two half pencils of elliptic curves with $E^2=F^2=0$ and $EF=1$, where $E$ and $F$ come from the two rulings of $\PP^1 \times \PP^1$. The linear system $|2L|$ maps $S$ to a quartic del Pezzo surface $D\subset \PP^4$ with four $A_1$ singularities, which coincides with the quotient $\PP^1 \times \PP^1 / \langle \iota \rangle$. We notice that the image of the branch curve $B\subset \PP^1 \times \PP^1$ under the quotient map is cut out on $D$ by a quadric, hence $S$ can be also viewed as a double covering of a 4-nodal quartic del Pezzo surface branched over a quadric section. Summarising we have the commutative diagram:
$$ \begin{tikzcd}
T \arrow{r}{2:1} \arrow{d}{/ \langle \sigma \rangle} & \PP^1 \times \PP^1 \arrow{d}{/ \langle \iota \rangle}\\
S \arrow{r}{|2L|} & D
\end{tikzcd} $$
In \cite{Hor781} Horikawa proved that a general Enriques surface admits a non-special polarization of degree 2, namely:
\begin{theorem}\cite[Theorem 4.1]{Hor781} \cite[Proposition VIII 18.1]{BHPV04}
Let $S$ be a general Enriques surface. Then there exists a $\iota$-invariant $(4,4)$-curve $B$ on $\PP^1 \times \PP^1$ such that the universal covering $T$ of $S$ is the
minimal resolution of the double covering of $\PP^1 \times \PP^1$ ramified over $B$. The curve $B$ is reduced with at worst simple singularities and does not contain any fixed point of $\iota$. The Enriques involution on $T$ is induced by the involution $\iota$ on $\PP^1\times \PP^1$.
\end{theorem} 

To obtain a representation of all Enriques surfaces, one still needs to treat the \textit{special} case (see \cite[Theorem 4.2]{Hor781} and \cite[Proposition VIII 18.2]{BHPV04}). In a similar way as above, one can construct an Enriques surface from a quadric cone in $\PP^3$ together with an involution. Indeed, the minimal resolution of the double covering of the cone branched over a curve cut out by a quartic polynomial is a K3 surface. The involution on the cone induces a fixed point free involution on the K3 surface, whose quotient is an Enriques surface. The hyperplane class of the cone induces a quasi-polarization $L$ of degree 2 on the Enriques surface, which splits as a sum $L=2E+R$, where $E$ is a half pencil of elliptic curves and $R$ is a $(-2)$-curve with $ER=1$. Notice that this degree 2 line bundle is big and nef, but not ample, as it is orthogonal to the class of $R$ coming from the resolution of the vertex of the cone.

By \cite{Hor781} every Enriques surface admits a special quasi-polarization or a non-special polarization of degree 2. In the present article, we will consider only the non-special polarization, as the general Enriques surface can be endowed with it.

\section{Background on GIT for degree two Enriques surfaces}
\label{sec:background}

Via Horikawa's model, one can construct a GIT compactification of the moduli space of non-special Enriques surfaces of degree 2 by looking at the isomorphism classes of branch curves $B$ on $\PP^1\times \PP^1$. The $ \iota $-invariant polynomials of bidegree $ (4,4) $ form a $ 13 $-dimensional vector space with a basis consisting of
$$ x_0^ix_1^{4-i}y_0^jy_1^{4-j} \ \mathrm{for} \ i+j\equiv 0 \ \mathrm{mod} \ 2,$$
which is explicitly
$$ x_0^{2k}x_1^{4-2k}y_0^{2l}y_1^{4-2l}, \ 0\leq k, l \leq 2, $$
$$ x_0^3x_1y_0^3y_1, \ x_0x_1^3y_0^3y_1, \ x_0^3x_1y_0y_1^3, \ x_0x_1^3y_0y_1^3. $$
We denote the corresponding linear system on $\PP^1 \times \PP^1$ by $ X:=\PP H^0(\PP^1\times \PP^1, \calO_{\PP^1\times \PP^1}(4,4))^{\iota}\cong \PP^{12} $.
	
Let $ G $ be the subgroup of the automorphism group of $ \PP^1 \times \PP^1 $, commuting with the involution $ \iota $ or, equivalently, fixing the set $ \Delta $. The group has dimension 2 and has the structure of semidirect product
$$ G=(\CC^*)^2\rtimes D_8, $$	
where $ D_8 $ is the dihedral group of symmetries of the square. The group $ D_8 =(\zz \times \zz) \ltimes \zz $ acts on $(\CC^*)^2$ as follows: the first two involutions act via inversion on every factor of the torus, while the third interchanges the two factors.
	
In \cite{Sha81} Shah describes explicitly the group $G< \Aut (\PP^1 \times \PP^1)$ and its action on $X$ in the following way. Let $ I_1 $ be the involution on $\PP^1$ which keeps $ x_0 $ fixed and sends $ x_1\mapsto -x_1 $, and let $ I_2 $ be the involution of $\PP^1$ which keeps $ y_0 $ fixed and sends $ y_1 \mapsto -y_1$. Let $ \gamma $ denote the automorphism of $ \PP^1 \times \PP^1 $ which interchanges the factors; let $ \langle \gamma \rangle $ be the group generated by $ \gamma $. For $ i=1,2 $ let $ G_i $ be the subgroup of $ \PGL(2, \CC) $ which commutes with $ I_i $: $ G_i $ is the stabiliser of the set of fixed points of $ I_i $. Then, $ G_1 $ and $ G_2 $ are isomorphic to the semidirect product $ \CC^* \rtimes \zz $, where $ \CC^* $ acts via the transformations
$$ (x_0: x_1) \mapsto (ax_0:a^{-1}x_1), \quad a\in \CC^*$$
$$ (y_0: y_1) \mapsto (by_0:b^{-1}y_1), \quad b\in \CC^* $$
The subgroup $ \zz $ is generated by the involution which interchanges $ x_0$ and $ x_1 $ as an element of $ G_1 $, and interchanges $ y_0$ and $ y_1 $ as an element of $ G_2 $. The group $ G $ is therefore isomorphic to
$$ G \cong (G_1\times G_2) \rtimes \langle \gamma \rangle. $$
	
We are now ready to construct the relevant moduli space of degree 2 Enriques surfaces. Geometric Invariant Theory  (\cite{MFK94}) provides a good categorical projective quotient
$$ M^{GIT}:=X/\!\!/ G, $$
which can be thought of as a compactification of the moduli space of non-special Enriques surfaces of degree 2. Via Horikawa's model of Section \S \ref{sec:Horikawa}, one can equivalently construct the same quotient by considering the linear system of quadric sections $|\calO_D(2)|\cong \PP^{12}$ on a 4-nodal del Pezzo surface $D \subset \PP^4$ modulo the action of the automorphism group of $D$, which is again isomorphic to $G$. 

By a lattice theoretical result \cite[Corollary 1.5.4.]{CDL20}, the non-special polarization of degree 2 constructed by Horikawa is the unique numerical polarization of degree 2 on an Enriques surface, up to an isometry of the Enriques lattice. Indeed, it is defined only up to numerical equivalence, since it is induced by an ample line bundle on the K3 covering. Therefore we have:
\begin{theorem} \cite[Theorem 5.8.5]{CDL20} \label{thm:polarized}
The GIT quotient $M^{GIT}$ is a compactification of the moduli space of numerically polarized Enriques surfaces of degree 2 and it is rational.
\end{theorem}

We aim at computing the intersection Betti numbers of $M^{GIT}$. We recall that intersection cohomology satisfies Poincar\'{e} duality, allowing us to compute the Betti numbers up to dimension $10=\dim M^{GIT}$. Hence we will report the results $mod \ t^{11}$ for the sake of readability. Nevertheless, we prefer to carry out the computations in all dimensions as a good way to double-check the calculations.

In order to find the intersection cohomology of $M^{GIT}$, we need to study the semistability conditions for the branch curves in $X$. We refer to \cite{MFK94} for the standard definitions of stability, semistability and polystability. In our case this description is provided by the following results of Shah \cite{Sha81}, which in turn come from the Hilbert-Mumford criterion (\cite{MFK94}). Here the four coordinate lines $ x_0=0 $, $ x_1=0 $, $ y_0=0 $, $ y_1=0 $ in $\PP^1 \times \PP^1$ are called \textit{edges}.

\begin{prop}\cite[Proposition 5.1.]{Sha81} 
A curve in $X$ is not semistable under the action of $G$ if and only if either it has a point of multiplicity greater than 4 (which must necessarily be in $\Delta$) or it has a quadruple point in $\Delta$ with an edge as a tangent of multiplicity greater than 3 at that point.
\end{prop}

\begin{prop}\cite[Proposition 5.2.(a)]{Sha81}
A curve in $X$ is strictly semistable (i.e. semistable, but not stable) under the action of $G$ if and only if either it has an edge as a component with multiplicity 2 or it has a quadruple point in $\Delta$.
\end{prop}

\begin{prop}\cite[Proposition 5.2.(b)]{Sha81}\label{prop:polystable}
The strictly polystable curves in $ X $ under the action of $ G $ fall into three categories:
\begin{enumerate}[(i)]
	\item Unions of two skew double edges and the components of the residual curve are mutually disjoint lines, none of which is an edge (see for example Figure \ref{fig:spscurves}(A)). 
	\item Unions of four $ \iota $-invariant curves of bidegree $ (1,1) $, each of which passes through two quadruple points in $ \Delta $. Moreover, these curves are not necessarily distinct and do not contain an edge as a component with multiplicity 2 (see for example Figure \ref{fig:spscurves}(B)).
	\item Union of all the edges with multiplicity 2 (see Figure \ref{fig:spscurves}(C)).
\end{enumerate}
\end{prop}

\begin{figure}[htbp]
\centering
\begin{subfigure}{0.3\textwidth}
\centering
\begin{tikzpicture}
\draw (1,1)  node[circle,fill,inner sep=2pt]{};
\draw (-1,1)  node[circle,fill,inner sep=2pt]{};
\draw (1,-1)  node[circle,fill,inner sep=2pt]{};
\draw (-1,-1)  node[circle,fill,inner sep=2pt]{};
\draw (-1, 1.05) -- (1, 1.05);
\draw (-1, 0.95) -- (1, 0.95);
\draw (-1, -1.05) -- (1, -1.05);
\draw (-1, -0.95) -- (1, -0.95);
\draw (0.75, -1.05) -- (0.75, 1.05);
\draw (0.35, -1.05) -- (0.35, 1.05);
\draw (-0.75, -1.05) -- (-0.75, 1.05);
\draw (-0.35, -1.05) -- (-0.35, 1.05);
\end{tikzpicture}
\caption{}
\end{subfigure}
\begin{subfigure}{0.3\textwidth}
\centering
\begin{tikzpicture}
\draw (1,1)  node[circle,fill,inner sep=2pt]{};
\draw (-1,1)  node[circle,fill,inner sep=2pt]{};
\draw (1,-1)  node[circle,fill,inner sep=2pt]{};
\draw (-1,-1)  node[circle,fill,inner sep=2pt]{};
\draw (1, 1) to[out=180, in=90] (-1, -1);
\draw (1, 1) to[out=-90, in=0] (-1, -1);
\draw (1, 1) to[out=202.5, in=67.5] (-1, -1);
\draw (1, 1) to[out=247.5, in=22.5] (-1, -1);
\end{tikzpicture}
\caption{}
\end{subfigure}
\begin{subfigure}{0.3\textwidth}
\centering
\begin{tikzpicture}
\draw (1,1)  node[circle,fill,inner sep=2pt]{};
\draw (-1,1)  node[circle,fill,inner sep=2pt]{};
\draw (1,-1)  node[circle,fill,inner sep=2pt]{};
\draw (-1,-1)  node[circle,fill,inner sep=2pt]{};
\draw (-1, 1.05) -- (1, 1.05);
\draw (-1, 0.95) -- (1, 0.95);
\draw (-1, -1.05) -- (1, -1.05);
\draw (-1, -0.95) -- (1, -0.95);
\draw (1.05, -1) -- (1.05, 1);
\draw (0.95, -1) -- (0.95, 1);
\draw (-1.05, -1) -- (-1.05, 1);
\draw (-0.95, -1) -- (-0.95, 1);
\end{tikzpicture}
\caption{}
\end{subfigure}
\caption{Strictly polystable curves}
\label{fig:spscurves}
\end{figure}
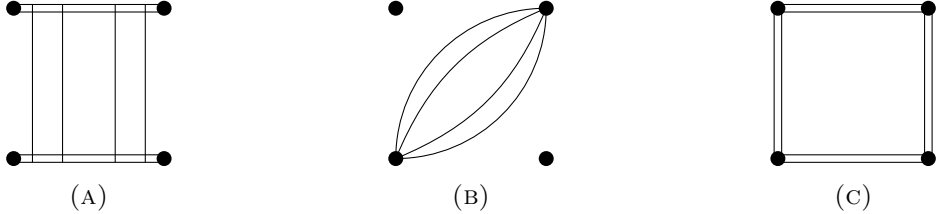

\begin{remark}
    Each family of strictly polystable points described in Proposition \ref{prop:polystable} (i) and (ii) consists of two disjoint irreducible components in $X$, which are interchanged by the action of the Weyl group of $G$. Every connected component is an open subset of a linear subspace of $X$. 
    Instead, the family of Proposition \ref{prop:polystable} (iii) consists of one point. We refer to Proposition \ref{prop:blow-up} for a description of these loci with respect to the coordinates of $X$.
\end{remark}

\section{Equivariant stratification} \label{sec:strati}
In this Section, we discuss the first step of Kirwan's method to compute the cohomology of GIT quotients. It consists of an equivariant stratification of the parameter space measuring the instability of every point under the group action (cf. Theorem \ref{thm:strata}). This stratification turns out to be perfect, in the sense that the Betti numbers of all strata sum up to the cohomology of the whole parameter space (cf. Theorem \ref{thm:equi}). We then apply these results to our case of numerically polarized Enriques surfaces of degree 2, and we obtain in Proposition \ref{prop:seriesss} the equivariant Betti numbers of the semistable locus.

\subsection{The HKKN stratification}\label{subsec:strata}
The first step in Kirwan's procedure (see \cite{Kir84}) is to consider the \textit{Hesselink-Kempf-Kirwan-Ness (HKKN) stratification} of the parameter space, which, from a symplectic viewpoint, coincides with the Morse stratification for the norm-square of an associated moment map. 

In general, let $ X\subset \PP^n $ be a complex projective manifold, acted on by a complex reductive group $ G $, inducing a linearization on the very ample line bundle $ L=\calO_{\PP^n}(1)|_{X} $. We pick a maximal compact subgroup $ K \subset G $, whose complexification gives $ G $, and a maximal torus $ T\subset G $, such that $ T\cap K $ is a maximal compact torus of $ K $. Before describing the stratification, we need also to fix an inner product together with the associated norm $ \| .\| $ on the dual Lie algebra $ \mathfrak{t}^{\vee}:=\mathrm{Lie}(T\cap K)^{\vee} $, e.g. the Killing form, invariant under the adjoint action of $ K $.
\begin{theorem}\label{thm:strata}\cite{Kir84}
	In the above setting, there exists a natural stratification of $ X$
	\begin{equation}
	X=\bigsqcup_{\beta \in \calB}S_\beta
	\end{equation}
	by $ G $-invariant locally closed subvarieties $ S_{\beta} $, indexed by a finite partially ordered set $ \calB\subset \mathrm{Lie}(T\cap K) $ such that the minimal stratum $ S_0=X^{ss} $ is the semistable locus of the action and the closure of $ S_{\beta} $ is contained in $ \bigcup_{\gamma \geq \beta} S_{\gamma}$, where $ \gamma \geq \beta $ if and only if $ \gamma = \beta $ or $ \|\gamma \| > \|\beta \| $. 
\end{theorem}

We briefly sketch the construction of the strata appearing in the previous Theorem \ref{thm:strata} (see \cite{Kir84} for more details). Let $ \{ \alpha_0,..., \alpha_n \}\subset \mathfrak{t}^{\vee} $ be the weights of the representation (a.k.a. the linearization) of $ G $ on $ \CC^{n+1} $ and identify $ \mathfrak{t}^{\vee} $ with $ \mathfrak{t} $ via the invariant inner product. After choosing a positive Weyl chamber $ \mathfrak{t}_{+} $, an element $ \beta\in \bar{\mathfrak{t}}_{+} $ belongs to the indexing set $ \calB $ of the stratification if and only if $ \beta $ is the closest point to the origin of the convex hull of some non-empty subset of $ \{ \alpha_0,..., \alpha_n \} $. We define $ Z_{\beta} $ to be the linear section of $ X $
$$ Z_{\beta}:=\{(x_0:...:x_n)\in X: x_i=0 \ \mathrm{if} \ \alpha_i.\beta\neq \|\beta \|^2 \}. $$
The stratum indexed by $ \beta $ is then
$$ S_{\beta}:=G\cdot \bar{Y}_{\beta}\smallsetminus \bigcup_{\|\gamma\|>\|\beta\|} G\cdot \bar{Y}_{\gamma}, $$
where
$$ \bar{Y}_{\beta}:=\{ (x_0:...:x_n)\in X: x_i=0 \ \mathrm{if} \ \alpha_i.\beta< \|\beta \|^2 \}. $$

The heart of Kirwan's results in \cite{Kir84} is the proof that the equivariant Betti numbers of the strata sum up to the cohomology of the whole space.

\begin{theorem}\label{thm:equi} \cite[8.12]{Kir84}
The stratification $ \{ S_\beta \}_{\beta \in \calB} $, constructed in Theorem \ref{thm:strata}, is $ G $-equivariantly perfect, namely the following holds:
	\begin{equation}
	P_{t}^{G}(X^{ss})=P_{t}^{G}(X)-\sum_{0\neq\beta \in \calB }t^{2 \mathrm{codim}(S_{\beta})}P_{t}^{G}(S_\beta).
	\end{equation}
\end{theorem}
\begin{remark}\label{rmk:unstable}
	If we denote by $ \Stab \beta \subset G $ the stabiliser of $ \beta \in \mathfrak{t} $ under the adjoint action of $ G $, the equivariant Hilbert-Poincar\'{e} cohomology of each stratum is 
	$$ P_t^G(S_{\beta})=P_t^{\Stab \beta}(Z_{\beta}^{ss}), $$
	where $ Z_{\beta}^{ss} $ is the set of semistable points of $ Z_{\beta} $ with respect to a suitable linearization of the action of $ \Stab \beta $ (cf. \cite[8.11]{Kir84}).
\end{remark}
\subsection{Stratification for degree 2 Enriques surfaces}
We now come back to our case described in the Section \S \ref{sec:background}. We apply Kirwan's results of the previous subsection to prove the following: 
	\begin{prop} \label{prop:seriesss}
	The $ G$-equivariant Hilbert-Poincar\'e series of the semistable locus is 
		\begin{align*}
		P_{t}^{G}(X^{ss})&=\frac{1+t^2+t^4+t^6+t^8+t^{10}+t^{12}-2t^{16}-3t^{18}-3t^{20}-2t^{22}+t^{26}+t^{28}+t^{30}+t^{32}}{(1-t^4)(1-t^8)}\\
		&\equiv P_{t}^{G}(X)\equiv 1+t^2+2t^4+2t^6+4t^8+4t^{10} \ \mathrm{mod} \ t^{11}. 
		\end{align*}
	\end{prop}
We need to start computing the equivariant Hilbert-Poincar\'{e} series $ P_{t}^{G}(X) $. Since $ X $ is compact, its equivariant cohomology ring is the invariant part under the action of $ \pi_0(G)=D_8 $ of $ H^*_{G^0}(X) $, which splits into the tensor product $ H^*(BG^0)\otimes H^*(X) $ (see \cite[8.12]{Kir84}). Then
\begin{equation}\label{eq:HGX}
    \begin{aligned}
    H^*_G(X)&= (H^*(\PP^{12})\otimes H^*(B(\CC^*)^2))^{D_8}\\
	& =(\QQ[h]/(h^{13})\otimes \QQ[c_1,c_2])^{D_8}.
    \end{aligned}
\end{equation}
In fact $ H^*(B(\CC)^2)\cong\QQ[c_1, c_2] $, where $ c_1 $ and $c_2$ have degree 2, and $ H^*(\PP^n)=\QQ[h]/(h^{n+1}) $, with $\deg(h)=2$. The group $ D_8 =(\zz \times \zz) \ltimes \zz $ acts on $(\CC^*)^2$ as follows: the first two involutions act via inversion $a \leftrightarrow a^{-1}$ on every factor of the torus, while the third interchanges the two factors. Moreover $D_8$ fixes the hyperplane class $ h\in H^2(\PP^{12}) $, as it acts on $\PP^{12}$ by change of coordinates. Therefore the ring of invariants is generated by $ c_1^2+c_2^2$, $ c_1^2c_2^2$ and $ h $:
$$ H^*_G(X)=\QQ[c_1^2+c_2^2, c_1^2c_2^2]\otimes \QQ[h]/(h^{13}). $$
Since $ \deg(c_1^2+c_2^2)=4$ and $\deg(c_1^2c_2^2)=8$, we have:
\begin{equation}\label{eq:PGX}
    P_t^G(X)=\dfrac{1+...+t^{24}}{(1-t^4)(1-t^8)}.
\end{equation}
According to Theorem \ref{thm:equi}, we need to subtract the contributions coming from the unstable strata. In our case, the indexing set $\calB $ of the stratification can be visualised by means of the Figure \ref{fig:weights}, called \textit{Hilbert diagram}.
	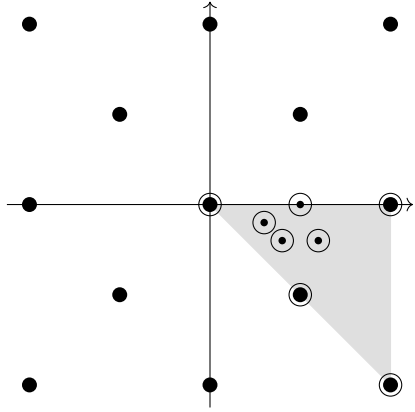
\begin{figure}[htbp] 
		\begin{tikzpicture}[scale=.6]
		\draw [lightgray!50, fill=lightgray!50] (0,0) -- (4,0) -- (4,-4) -- (0,0);
		\draw [->] (-4.5,0) -- (4.5,0);
		\draw [->] (0, -4.5) -- (0, 4.5);
		
		\draw (-4,-4) node[circle,fill, inner sep=2pt]{};
		\draw (-4,0) node[circle,fill, inner sep=2pt]{};
		\draw (-4,4) node[circle,fill, inner sep=2pt]{};
		
		\draw (-2,-2) node[circle,fill, inner sep=2pt]{};
		\draw (-2,2) node[circle,fill, inner sep=2pt]{};
		
		\draw (0,-4) node[circle,fill, inner sep=2pt]{};
		\draw (0,0) node[circle,fill, inner sep=2pt]{};
		\draw (0,4) node[circle,fill, inner sep=2pt]{};
		
		\draw (2,-2) node[circle,fill, inner sep=2pt]{};
		\draw (2,2) node[circle,fill, inner sep=2pt]{};
		
		\draw (4,-4) node[circle,fill, inner sep=2pt]{};
		\draw (4,0) node[circle,fill, inner sep=2pt]{};
		\draw (4,4) node[circle,fill, inner sep=2pt]{};

		\draw (0,0)  node[circle,fill,inner sep=1pt]{};
		\draw (0,0)  node[circle,draw, inner sep=3pt]{};
		\draw (4,0)  node[circle,fill,inner sep=1pt]{};
		\draw (4,0)  node[circle,draw, inner sep=3pt]{};
		\draw (4,-4)  node[circle,fill,inner sep=1pt]{};
		\draw (4,-4)  node[circle,draw, inner sep=3pt]{};
		\draw (2,-2)  node[circle,fill,inner sep=1pt]{};
		\draw (2,-2)  node[circle,draw, inner sep=3pt]{};
		\draw (2,0)  node[circle,fill,inner sep=1pt]{};
		\draw (2,0)  node[circle,draw, inner sep=3pt]{};
		\draw (6/5,-2/5)  node[circle,fill,inner sep=1pt]{};
		\draw (6/5,-2/5)  node[circle,draw, inner sep=3pt]{};
		\draw (8/5,-4/5)  node[circle,fill,inner sep=1pt]{};
		\draw (8/5,-4/5)  node[circle,draw, inner sep=3pt]{};
		\draw (12/5,-4/5)  node[circle,fill,inner sep=1pt]{};
		\draw (12/5,-4/5)  node[circle,draw, inner sep=3pt]{};
		\end{tikzpicture}
		\caption{\textit{Hilbert diagram}.}\label{fig:weights}
	\end{figure}

There are 13 black nodes in this square, and each of these nodes represents a monomial 
$$ x_0^ix_1^{4-i}y_0^jy_1^{4-j} \ \mathrm{for} \ i+j\equiv 0 \ \mathrm{mod} \ 2$$
in $ H^0(\PP^1\times\PP^1, \calO_{\PP^1\times\PP^1}(4, 4))^{\iota} $. This square is simply the diagram of weights $ \alpha_I=\alpha_{(i,j)} $ of the representation of $ G $ on $H^0(\PP^1\times\PP^1, \calO_{\PP^1\times\PP^1}(4, 4))^{\iota}$ with respect to the standard maximal torus $ T:=(\diag(a, a^{-1}), \diag(b, b^{-1}), 1) $ in $ G $. Each of the nodes denotes a weight of this representation, namely
\begin{equation}\label{eq:weight}
x_0^ix_1^{4-i}y_0^jy_1^{4-j} \leftrightarrow (4-2i, 4-2j), \ \mathrm{for} \ i+j\equiv 0 \ \mathrm{mod} \ 2.
\end{equation}
There is a non-degenerate inner product (the Killing form) defined on the Lie algebra $\mathfrak{t}:=\text{Lie}(T) \cong \text{Lie}(G) $. Using this inner product, we can identify the Lie algebra $ \mathfrak{t} $ with its dual $ \mathfrak{t}^\vee $, and the above square can be thought of as lying in $ \mathfrak{t} $. The axes of the Hilbert diagram thus coincide with the Lie algebras of the two factors of the maximal compact torus.

The Weyl group $ W(G):=N(T)/T\cong (\zz \times \zz) \ltimes \zz$ coincides with the dihedral group $ D_8 $ of all symmetries of the square. It operates on the \textit{Hilbert diagram} as follows: the first two involutions are reflections along the axes, while the third one is along the principal diagonal. It is easy to see that the grey region is the portion of the square which lies inside a fixed positive Weyl chamber $ \mathfrak{t}_+ $.

By definition, the indexing set $ \calB $ consists of vectors $ \beta $ such that $ \beta $ lies in the closure $ \bar{\mathfrak{t}}_+ $ of the positive Weyl chamber and is also the closest point to the origin of a convex hull spanned by a non-empty set of weights of the representation of $ G $ on $ H^0(\PP^1\times\PP^1, \calO_{\PP^1\times\PP^1}(4, 4))^{\iota} $. In this situation, we may assume that such a convex hull is either a single weight or it is cut out by a line segment joining two weights, which will be denoted by $ \langle \beta \rangle $ (see Figure \ref{fig:weights}).

The codimension $ d(\beta)$ of each stratum $ S_{\beta}\subset X $ is equal to (see \cite[3.1]{Kir89})
\begin{equation}\label{codim}
 d(\beta)= n(\beta)-\dim G/P_\beta, 
\end{equation} 
where $ n(\beta) $ is the number of weights $ \alpha_I $ such that $ \beta\cdot \alpha_I < ||\beta||^2 $, i.e. the number of weights lying in the half-plane containing the origin and defined by $ \beta $. Moreover, let $ P_\beta \subseteq G$ be the subgroup of elements in $ G $ which preserve $ \bar{Y}_{\beta} $, then $ P_{\beta} $ is a parabolic subgroup, whose Levi component is the stabiliser $ \Stab\beta $ of $ \beta\in \mathfrak{t} $ under the adjoint action of $ G $. In our case, the parabolic subgroup $P_{\beta}\cong T$ has always dimension 2, hence $ \dim(G/P_{\beta})=0 $.

All the contributions coming from the unstable strata are summarised in Table \ref{tab:1} and can be deduced by analysing Figure \ref{fig:weights}.

	\begin{table}[htpb]
		\centering
		\begin{tabular}{c|c|c|c|c}
			weights in $ \langle \beta \rangle $& $ n(\beta) $&$\Stab \beta $& $2d(\beta)$&$P_t^{G}(S_{\beta})$\\ 
			\hline \rule{0pt}{2.5ex}$ (4, -4) $& 12 &$ (\CC^*)^2\rtimes \zz $&24&$ (1-t^2)^{-1}(1-t^4)^{-1} $\\[1ex]
			$ (4, 0), (2, -2), (0, -4) $& 9 &$ (\CC^*)^2\rtimes \zz $&18&$ (1+t^2-t^6)(1-t^2)^{-1}(1-t^4)^{-1} $\\[1ex]
			$ (4,4),(2,-2) $& 9 &$ (\CC^*)^2 $&18&$ (1-t^2)^{-1}$\\[1ex]
			$ (2,2),(0, -4) $& 7 &$ (\CC^*)^2 $&14&$ (1-t^2)^{-1}$\\[1ex]
			$ (4,4), (0, -4) $& 8 &$ (\CC^*)^2 $&16&$ (1-t^2)^{-1} $\\[1ex]
			$ (2, 2), (2, -2) $& 8 &$ \CC^*\times G_2 $&16&$ (1-t^2)^{-1} $\\[1ex]
			$ (4, 4), (4, 0), (4, -4) $& 10 &$ \CC^*\times G_2 $&20&$ (1+t^2-t^6)(1-t^2)^{-1}(1-t^4)^{-1} $\\[1ex]
		\end{tabular}
		\caption{Cohomology of the unstable strata.} \label{tab:1}
	\end{table}
 For every $ \beta \in \calB $, the first column of Table \ref{tab:1} shows the weights contained in the segment $ \langle \beta \rangle $ orthogonal to the vector $ \beta\in \mathfrak{t} $ (see Figure \ref{fig:weights}): then via the correspondence (\ref{eq:weight}) one can obtain an explicit geometrical interpretation of the curve contained in each unstable stratum. The terms appearing in the second, third and fourth columns are determined easily from the \textit{Hilbert diagram}. Here $ (\CC^*)^2\rtimes \mathbb{Z}_2 $ is a double covering of the maximal torus of $G $, determined by the extension $ (a, b)\leftrightarrow (b^{-1}, a^{-1}) $. The computations in the last column follow from applying Theorem \ref{thm:equi} to the action of $ \Stab \beta $ on $ Z_{\beta} $, in order to compute the equivariant cohomology of each unstable stratum $ P_t^{\Stab \beta}(Z_{\beta}^{ss})=P_t^G(S_{\beta}) $ (see Remark \ref{rmk:unstable}). 
 
 We shall discuss all the these cases below.
 \begin{lemma}\label{lem:unstable1}
	There is exactly one unstable stratum indexed by $\beta$, as listed in Table \ref{tab:1}, such that $Z_{\beta}\cong \PP^0$, and their equivariant Hilbert-Poincar\'{e} series is $P_t^G(S_{\beta})=(1-t^2)^{-1}(1-t^4)^{-1}$. 
\end{lemma}

\begin{proof}
The case under consideration corresponds to the first row of Table \ref{tab:1}, where the line orthogonal to $ \beta $ contains only the weight $\beta$ itself, giving the point $ Z_{\beta}\cong \PP^0 $. Hence by Remark \ref{rmk:unstable} the equivariant cohomology of the corresponding stratum is
	$$ P_t^G(S_{\beta})=P_t^{(\CC^*)^2\rtimes\zz}(\PP^0)=\frac{1}{(1-t^2)(1-t^4)}. $$
\end{proof}
 
\begin{lemma}\label{lem:unstable2}
	There are exactly four unstable strata indexed by $\beta$, as listed in Table \ref{tab:1}, such that $Z_{\beta}\cong \PP^1$, and their equivariant Hilbert-Poincar\'{e} series is $P_t^G(S_{\beta})=(1-t^2)^{-1}$. 
\end{lemma}

\begin{proof}
	Looking at Figure \ref{fig:weights}, there are four unstable strata indexed by $\beta \in \calB$ such that the segment $\langle \beta \rangle$ orthogonal to the vector $\beta$ contains two weights that generate the line $Z_{\beta}\subset X$. As summarised in Table \ref{tab:1}, in three of these cases the stabiliser $\Stab \beta$ is isomorphic to the maximal torus $(\CC^*)^2$ and hence by Remark \ref{rmk:unstable}
	$$ P_t^G(S_{\beta})=\frac{1+t^2}{(1-t^2)^2}-\frac{2t^2}{(1-t^2)^2}=\frac{1}{1-t^2}. $$
	In the remaining case, corresponding to the sixth row of Table \ref{tab:1}, the stabiliser is $\Stab \beta \cong \CC^* \times G_2$ and the cohomology of the corresponding stratum is 
	$$ P_t^G(S_{\beta})=\frac{1+t^2}{(1-t^2)(1-t^4)}-\frac{t^2}{(1-t^2)^2}=\frac{1}{1-t^2}. $$
\end{proof}

\begin{lemma}\label{lem:unstable3}
	There are exactly two unstable strata indexed by $\beta$, as listed in Table \ref{tab:1}, such that $Z_{\beta}\cong \PP^2$, and its equivariant Hilbert-Poincar\'{e} series is $P_t^G(S_{\beta})=(1+t^2-t^6)(1-t^2)^{-1}(1-t^4)^{-1}$. 
\end{lemma}

\begin{proof}
	The cases under consideration correspond to the second and last row of Table \ref{tab:1}, where the segment orthogonal to $ \beta $ contains three weights spanning $ Z_{\beta}\cong \PP^2 $. By Theorem \ref{thm:equi} the equivariant cohomological series of the correspondent stratum is
	$$ P_t^G(S_{\beta})=\frac{1+t^2+t^4}{(1-t^2)(1-t^4)}-\frac{t^4}{(1-t^2)^2}=\frac{1+t^2-t^6}{(1-t^2)(1-t^4)}. $$
\end{proof}
 
 We are finally ready to prove Proposition \ref{prop:seriesss}:
\begin{proof}[Proof of Proposition \ref{prop:seriesss}]
According to Theorem \ref{thm:equi}, we need to subtract all the contributions of the unstable strata, appearing in Table \ref{tab:1}, to the $ G $-equivariant cohomology of $ X $ computed in (\ref{eq:PGX}).
\end{proof}
 
\section{Kirwan blow-up}
\label{sec:blow}

In this section we recall the general construction of the Kirwan blow-up of a GIT quotient which provides an orbifold resolution of singularities. It is achieved by stratifying the GIT boundary $X/\!\!/G \smallsetminus X^s/\!\!/G$ in terms of the connected components $R$ of the stabilisers of the associated polystable orbits. Then, one proceeds by blowing up these strata according to the dimension of the corresponding $R$. In our situation, the Kirwan blow-up $M^K \rightarrow M$ is obtained by blowing up three loci of strictly polystable points, geometrically described in Proposition \ref{prop:polystable} (see also Proposition \ref{prop:blow-up}).

\subsection{General setting}
\label{subsec:blow}
In general the equivariant cohomology $ H^*_G(X^{ss}) $ of the semistable locus does not coincide with the cohomology $ H^*(X/\!\!/G) $ of the GIT quotient, unless in the case when all semistable points are actually stable. This is not the case for us. The solution is given by constructing a \textit{partial resolution of singularities} $ \widetilde{X}/\!\!/G \rightarrow X/\!\!/G $, known as \textit{Kirwan blow-up} \cite{Kir85}, for which the group $ G $ acts with finite isotropy groups on the semistable points $ \widetilde{X}^{ss} $. We briefly describe how it is constructed.

We consider again the setting, as in Section \S \ref{subsec:strata}, of a smooth projective manifold $ X \subset \PP^n $ acted on by a  reductive group $ G $. We also assume throughout the paper that the stable locus $ X^{s}\neq \varnothing $ is non-empty. In order to produce the Kirwan blow-up, we need to study the GIT boundary $ X/\!\!/G \smallsetminus X^s/\!\!/G $ and stratify it in terms of the isotropy groups of the associated semistable points. More precisely, let $ \calR $ be a set of representatives for the conjugacy classes of connected components of stabilisers of strictly polystable points, i.e. semistable points with closed orbits, but infinite stabilisers. Let $ r $ be the maximal dimension of the groups in $ \calR $, and let $ \calR(r)\subseteq \calR $ be the set of representatives for conjugacy classes of subgroups of dimension $ r $. For every $ R\in \calR(r) $, consider the fixed locus
\begin{equation}
	Z_{R}^{ss}:=\{ x \in X^{ss} : R \ \text{fixes} \ x \} \subset X^{ss}.
\end{equation}

Kirwan showed \cite[\S 5]{Kir85} that the subset 
$$ \bigcup_{R\in \calR(r)}G\cdot Z_{R}^{ss}\subset X^{ss} $$
is a disjoint union of smooth $ G $-invariant closed subvarieties in $ X^{ss} $. Now let $ \pi_1:X_1\rightarrow X^{ss} $ be the blow-up of $ X^{ss} $ along $\bigcup_{R\in \calR(r)} G\cdot Z_{R}^{ss} $ and $ E\subset X_1 $ be the exceptional divisor.

Since the centre of the blow-up is invariant under $ G $, there is an induced action of $ G $ on $ X_1 $, linearised by a suitable ample line bundle. If $ L=\calO_{\PP^n}(1)|_{X} $ is the very ample line bundle on $ X $ linearised by $ G $, then there exists $ d\gg 0 $ such that $ L_1:=\pi_1^* L^{\otimes d} \otimes \calO_{X_1}(-E) $ is very ample and admits a $ G $-linearization (see \cite[3.11]{Kir85}). After making this choice, the set $ \calR_1 $ of representatives for the conjugacy classes of connected components of isotropy groups of strictly polystable points in $ X_1 $ will be strictly contained in $ \calR $ (see \cite[6.1]{Kir85}). Moreover, the maximum among the dimensions of the reductive subgroups in $ \calR_1 $ is strictly less than $ r $. Now we restrict to the new semistable locus $ X_1^{ss}\subset X_1 $, so that we are ready to perform the same process as above again. 

After at most $ r $ steps, we obtain a finite sequence of \textit{modifications}:
\begin{equation}
	\widetilde{X}^{ss}:= X^{ss}_r \rightarrow ... \rightarrow X^{ss}_1 \rightarrow X^{ss},
\end{equation}
by iteratively restricting to the semistable locus and blowing up smooth invariant centres (cf. \cite[6.3]{Kir85}).

Therefore, in the last step, $ \widetilde{X}^{ss} $ is equipped with a $ G $-linearised ample line bundle such that $ G $ acts with finite stabilisers. In conclusion, we have the diagram
\begin{equation}
\begin{tikzcd}
\widetilde{X}^{ss} \arrow{r} \arrow{d} & X^{ss} \arrow{d} \\
\widetilde{X}/\!\!/G \arrow{r}& X/\!\!/G,
\end{tikzcd}
\end{equation}
where the \textit{Kirwan blow-up} $ \widetilde{X}/\!\!/G $, having at most finite quotient singularities, gives a \textit{partial desingularization} of $ X/\!\!/G $, which in general has worse singularities.

\subsection{Kirwan blow-up for degree 2 Enriques surfaces}
Coming back to our case, we need to find the indexing set $ \calR $ of the Kirwan blow-up and the corresponding spaces $ Z_{R}^{ss} $, for all $ R\in \calR $. Namely, one must compute the conjugacy classes of the connected components of the identity in the stabilisers among all three families of polystable curves listed in Proposition \ref{prop:polystable}. 

We must find which non-trivial connected reductive subgroups $ R \subset G $ fix at least one semistable point. Firstly, since $ R $ is connected, $ R $ must be contained in $ G^0=(\CC^*)^2 $, therefore it is a subtorus of rank 1 or 2. Secondly, since we are interested only in the conjugacy class of $R$, we may assume that its intersection $R\cap (S^1)^2$  with the maximal compact torus is a maximal compact subgroup. The fixed point set $ Z_R^{ss} $ in $X^{ss}  $ consists of all semistable points whose representatives in $ H^0(\PP^1\times\PP^1, \calO_{\PP^1\times\PP^1}(4, 4))^{\iota}\cong\CC^{13} $ are fixed by the linear action of $ R $.

If $R$ has rank 2, then it coincides with the whole $(\CC^*)^2$ and clearly $Z_R=\{ x_0^2x_1^2y_0^2y_1^2 \}$. Instead, if $R$ has rank 1, $ Z_R $ is spanned by those weight vectors which lie on a line through the centre of the \textit{Hilbert diagram} and orthogonal to the Lie subalgebra $ \text{Lie}(R) \subset \mathfrak{t} $. Up to the action of a suitable element of the Weyl group $ W(G) $, we can assume that the line passes through the chosen closed positive Weyl chamber $ \bar{\mathfrak{t}}_+ $. We have only two possibilities (see Figure \ref{fig:weights}), namely the x-axis and the bisector of the II and III quadrants. This considerations lead to the following:
	
\begin{prop} \label{prop:blow-up}
The indexing set of Kirwan blow-up, such as the fixed loci $ Z_R^{ss} $ for $ \iota $-invariant $ (4,4) $ curves in $ \PP^1 \times \PP^1 $ can be described as follows:
	\begin{enumerate}[(i)]
		\item $ R_0=G^0=(\CC^*)^2 $ and in this case
		$$Z_{R_0}=Z_{R_0}^{ss}=G \cdot Z_{R_0}^{ss}=\{ x_0^2x_1^2y_0^2y_1^2 \}; $$
		\item $ R_1=\{ (t,t)\in G^0: t\in \CC^* \}\cong \CC^* $ and in this case
		$$ Z_{R_1}=\PP\{ Ax_0^4y_1^4+Bx_0^3x_1y_0y_1^3+Cx_0^2x_1^2y_0^2y_1^2+Dx_0x_1^3y_0^3y_1+Ex_1^4y_0^4 \}=\PP(\CC[x_0y_1, x_1y_0]_4)\cong \PP^4, $$
		$$ Z_{R_1}^{ss}=\PP^4 \smallsetminus \{ A=B=C=0, \ C=D=E=0 \}, $$
		\begin{align*}
		G\cdot Z_{R_1}^{ss}= Z_{R_1}^{ss} \cup  \PP\{ A'x_0^4y_0^4+B'x_0^3x_1y_0^3y_1+C'x_0^2x_1^2y_0^2y_1^2+D'x_0x_1^3y_0y_1^3+E'x_1^4y_1^4 \} \\
		 \smallsetminus \{ A'=B'=C'=0, \ C'=D'=E'=0 \};
		\end{align*}
		\item $ R_2=\{ (1,t)\in G^0: t\in \CC^* \}\cong \CC^* $ and in this case
		$$ Z_{R_2}=\PP\{ ax_0^4y_0^2y_1^2+bx_0^2x_1^2y_0^2y_1^2+cx_1^4y_0^2y_1^2 \}=\PP(y_0^2y_1^2 \cdot \CC[x_0^2, x_1^2]_2)\cong \PP^2 $$
		$$ Z_{R_2}^{ss}=\PP^2 \smallsetminus \{ a=b=0,\ b=c=0 \}, $$
		$$ G \cdot Z_{R_2}^{ss}=Z_{R_2}^{ss} \cup \PP\{ a'x_0^2x_1^2y_0^4+b'x_0^2x_1^2y_0^2y_1^2+c'x_0^2x_1^2y_1^4 \} \smallsetminus \{ a'=b'=0,\ b'=c'=0 \}. $$
	\end{enumerate}
Moreover, the following holds:
$$ G\cdot Z_{R_1}^{ss} \cap G \cdot Z_{R_2}^{ss}=Z_{R_0}. $$
\end{prop}

We recall that Kirwan's partial desingularization process consists of successively blowing up $ X^{ss} $ along the (strict transforms of the) loci $ G\cdot Z_R^{ss} $ in order of dim$ R $, to obtain the space $ \widetilde{X}^{ss} $, and then taking the induced GIT quotient $ \widetilde{X}/\!\!/G $ with respect to a suitable linearization. 

In our situation, we get the diagram
$$ \begin{tikzcd}
\widetilde{X}^{ss}=(\Bl_{G\cdot Z_{R_2,1}^{ss}}X_2^{ss})^{ss} \arrow{r} \arrow{d}& X_2^{ss}=(\Bl_{G\cdot Z_{R_1,1}^{ss}}X_1^{ss})^{ss} \arrow{r} & X_1^{ss}=(\Bl_{Z_{R_0}}X^{ss})^{ss} \arrow{r} & X^{ss} \arrow{d} \\
M^K \arrow{rrr}& && M^{GIT} .
\end{tikzcd} $$

 The space $ \widetilde{X}^{ss} $ is obtained by  blowing up firstly the point $ Z_{R_0}^{ss} $, followed by the blow-up of $ G\cdot Z_{R_1,1}^{ss} $, namely the strict transform of the locus $ G\cdot Z_{R_1}^{ss} $ under the first bow-up. In the end we need to blow-up the strict transform $ G\cdot Z_{R_2, 1}^{ss} $ of the orbit $ G\cdot Z_{R_2}^{ss} $. We also observe that the third blow-up commutes with the second one, because the strict transforms
$$ G\cdot Z_{R_1,1}^{ss} \cap G\cdot Z_{R_2, 1}^{ss} = \varnothing $$
are disjoint. Thus we find:

\begin{definition}\label{def:Mtilde}
The \textit{Kirwan blow-up} $ M^K:=\widetilde{X}/\!\!/G \rightarrow M^{GIT} $ is defined as the GIT quotient of the blown-up variety $ \widetilde{X}^{ss} $ constructed above.
\end{definition}

 Intrinsically at the level of moduli spaces, $ M^K $ is obtained by first blowing up the point $ G\cdot Z_{R_0}/\!\!/G $ corresponding to the union of the four double edges (cf. Proposition \ref{prop:polystable} (iii)). Then one needs to blow-up the strict transform $ \Bl_{G\cdot Z_{R_0}/\!\!/G}(G\cdot Z_{R_1} /\!\!/G) $ of the threefold parametrizing the unions of four conics (cf. Proposition \ref{prop:polystable} (ii)). Eventually the blow-up of the curve $ \Bl_{G\cdot Z_{R_0}/\!\!/G}(G\cdot Z_{R_2} /\!\!/G) $ corresponding to the union of two skew double edges and two skew lines (cf. Proposition \ref{prop:polystable} (i)) completes the construction of $M^K$. Nevertheless, for computational reasons, we will prefer the description at the level of parameter spaces.

\section{Cohomology of the Kirwan blow-up}

This Section is devoted to the proof of

\begin{theorem}\label{thm:cohoblow}
	The Hilbert-Poincar\'{e} polynomial of the Kirwan blow-up $ M^K $ is
	$$ P_t(M^K)=1+4t^2+8t^4+13t^6+18t^8+20t^{10}+18t^{12}+13t^{14}+8t^{16}+4t^{18}+t^{20}.$$
\end{theorem}

In the first part of the Section, we recall the general theory to compute the Betti numbers of the Kirwan blow-up $\widetilde{X}/\!\!/G \rightarrow X/\!\!/G$ of a GIT quotient. Since $\widetilde{X}/\!\!/G$ has only finite quotient singularities, its rational cohomology coincides with the equivariant cohomology of the semistable locus $\widetilde{X}^{ss}$, which in turn can be computed from the equivariant cohomology of $X^{ss}$ corrected by an error term (see Theorem \ref{thm:cohkirblow}). This error term is divided into a main and extra contribution: the former takes into account the geometry of the centres of the blow-ups and the latter the action of $G$ on the exceptional divisors.

In the second and third part of this Section, we complete the computation of the Betti numbers of $M^K$ by calculating the main and extra terms appearing in Theorem \ref{thm:cohkirblow} for our case. This concludes the proof of Theorem \ref{thm:cohoblow}.  
 
\label{sec:cohomology}
\subsection{General setting} \label{subsec:settingblowup} Kirwan in \cite{Kir85} explained the effect of the desingularization on the equivariant Poincar\'{e} series. We consider again the setting, as in Section \S \ref{subsec:blow}, of a nonsingular projective variety $ X $ together with an action of a reductive group $ G $ and a $G$-linearization. Assume that $ R $ is a connected reductive subgroup with the property that the semistable fixed point set $ Z_R^{ss}\subset X^{ss} $ is non-empty, but that $ Z_{R'}^{ss}=\varnothing $ for all subgroups $ R'\subset G $ of higher dimension than $ R $.

Let $ \pi:\hat{X}\rightarrow X^{ss}$ be the blow-up of $ X^{ss} $ along $ G\cdot Z_R^{ss} $. Then the equivariant cohomology of $ \hat{X} $ is related to that of the exceptional divisor $ E $ by the formula:
\begin{equation}\label{cohomologyblowup}
H^*_G(\hat{X})=H^*_G(X^{ss})\oplus H^*_G(E)/H^*_G(G\cdot Z_R^{ss})
\end{equation}
(see \cite[\S 4.6]{GH78}, \cite[7.2]{Kir85}). If $ \calN^R $ denotes the normal bundle to $ G\cdot Z_R^{ss} $ in $ X^{ss} $, then the equivariant cohomology of the exceptional divisor $ E=\PP \calN^R $ can be computed via a degenerating spectral sequence, namely
$$ H^*_G(E)=H^*_G(G\cdot Z_R^{ss})(1+...+t^{2(\rk \calN^R-1)}). $$
Kirwan proved (\cite[5.10]{Kir85}) that $ G\cdot Z_{R}^{ss} $ is algebraically isomorphic to $ G\times_{N(R)} Z_{R}^{ss} $, where $ N(R)\subset G $ is the normaliser of $ R$, hence we can calculate
\begin{equation}\label{formula:rank}
\rk \calN^R=\dim X -\dim G\cdot Z_{R}^{ss}= \dim X-(\dim G+\dim Z_{R}^{ss} -\dim N(R))
\end{equation}
and
$$ H^*_G(G\cdot Z_R^{ss})=H^*_{N(R)}(Z_R^{ss}). $$ 
Therefore from (\ref{cohomologyblowup}), it follows that
$$ P_t^G(\hat{X})=P_t^G(X^{ss})+P_t^{N(R)}(Z_R^{ss})(t^2+...+t^{2(\rk \calN^R-1)}).$$
By considering the \textit{HKKN stratification} $ \{ S_{\beta} \}_{\beta \in \calB} $ associated to the induced action of $ G $ on $ \hat{X} $ (see Theorem \ref{thm:strata}), we can apply Theorem \ref{thm:equi} to deduce the equivariant Hilbert-Poincar\'{e} series of the semistable locus:
\begin{equation}
P_{t}^{G}(\hat{X}^{ss})=P_{t}^{G}(\hat{X})-\sum_{0\neq\beta \in \hat{\calB} }t^{2 \mathrm{codim}(\hat{S}_{\beta})}P_{t}^{G}(\hat{S}_\beta).
\end{equation}
In order to apply this formula, we have to determine the indexing set $ \hat{\calB} $. For this, we choose a point $ x\in Z_R^{ss} $ and consider the normal vector space $ \calN_x^R $ to $ G\cdot Z_R^{ss} $ in $ X^{ss} $ at this point. Since the action of $ R $ on $ X^{ss} $ leaves this point $ x $ fixed, there is a natural induced representation $ \rho: R \rightarrow \GL(\calN_x^R) $ of $ R $ on this vector space. Let $ \calB(\rho) $ denote the indexing set of the stratification of the $ R $-action on the projective normal slice $ \PP \calN_x^R $. For each $ \beta' \in \calB(\rho) $, we have the subspaces $ Z_{\beta', \rho}$, $ Z_{\beta', \rho}^{ss}$ and $ S_{\beta', \rho} $ defined as in Section \S \ref{subsec:strata} but with respect to the action of $ R $ on $ \PP \calN_x^R $. 

In \cite[\S 7]{Kir85}, it is proved that $ \hat{\calB} $ can be identified with a subset of $ \calB(\rho) $. Given $ \beta \in \hat{\calB} $, the Weyl group orbit $ W(G) $ of $ \beta $ decomposes into a finite number of $ W(R) $ orbits. There is a unique $ \beta'\in \calB(\rho) $ in each $ W(R) $ orbit contained in the $ W(G) $ orbit of $ \beta $. We thus denote by $ w(\beta', R, G) $ the number of $ \beta' \in \calB(\rho) $ lying in the Weyl group orbit $ W(G)\cdot \beta $. 

For each $ \beta' \in \hat{\calB} \subset \calB(\rho) $, there is an $ (N(R)\cap \Stab \beta') $-equivariant fibration
$$ \pi:Z_{\beta', R}^{ss}:=Z_{\beta'}^{ss}\cap \pi^{-1}(Z_R^{ss})\rightarrow Z_R^{ss} $$
 with all fibres isomorphic to $ Z_{\beta', \rho}^{ss} $. Like for each stratum $ \hat{S}_{\beta'} $, its codimension in $ \hat{X} $ is the same as the codimension of $ \hat{S}_{\beta', \rho} $ in $ \PP \calN_x^R $, denoted by $ d(\PP \calN^R, \beta') $ and its Hilbert-Poincar\'{e} series $ P^G_t(\hat{S}_{\beta'}) $ is the same as $ P_t^{N(R)\cap \Stab \beta}(Z_{\beta', R}^{ss}) $.

A repeated application of this argument leads to a formula to compute inductively the equivariant cohomology $ H^*_G(\widetilde{X}^{ss}) $ of the semistable locus $ \widetilde{X}^{ss} $, whose GIT quotient gives the \textit{Kirwan blow-up}. Since $ G $ acts on $ \widetilde{X}^{ss} $ with finite stabilisers, its equivariant Hilbert-Poincar\'{e} polynomial coincides with that of the partial desingularization $ \widetilde{X}/\!\!/G $. We summarise all the previous theory under the following
\begin{theorem}\cite[7.4]{Kir85} \label{thm:cohkirblow}In the above setting,
	the cohomology of the Kirwan blow-up is given by:
	$$ P_t(\widetilde{X}/\!\!/G)=P_{t}^G(\widetilde{X}^{ss})=P_{t}^G(X^{ss})+\sum_{R\in\calR}A_{R}(t), $$
	where the error term $ A_{R}(t) $ can be divided into main and extra terms, as follows:
\begin{align}
A_R(t) = & \ P_t^{N}(Z_R^{ss})(t^2+...+t^{2(\rk \calN^R-1)}) \tag{Main term}\\
&-\sum_{0\neq \beta'\in\calB(\rho)}\frac{1}{w(\beta', R, G)}t^{2d(\PP\calN^R, \beta')}P_t^{N\cap \Stab \beta'}(Z_{\beta', R}^{ss}). \tag{Extra term}
\end{align}
\end{theorem}
 \begin{remark} \label{rmk:cohextra}(cf. \cite[7.2]{Kir85} and \cite[4.1 (4)]{Kir88}) If $ Z_{\beta', \rho}^{ss}=Z_{\beta', \rho} $, the spectral sequence of rational equivariant cohomology associated to the fibration $ \pi:Z_{\beta',R}^{ss}\rightarrow Z_R^{ss} $ degenerates and hence
 	$$ P_t^{N\cap \Stab \beta'}(Z_{\beta', R}^{ss})=P_t^{N\cap \Stab \beta'}(Z_{R}^{ss})\cdot P_t(Z_{\beta', \rho}). $$
 \end{remark}

Due to the role they play in the aforementioned results, we compute the normalisers of the reductive subgroups in $ \calR $.

\begin{prop} \label{prop:normal}
	The normalisers of the reductive subgroups in $ \calR=\{ R_0, R_1, R_2 \} $ (see Proposition \ref{prop:blow-up}) are given as follows:
	\begin{enumerate}[(i)]
		\item $ N(R_0)=G $;
		\item $ N(R_1)= (\CC^*)^2 \rtimes (\zz \times \zz) $ with action of the first $ \zz $ by $ (a, b)\leftrightarrow (b, a) $ and the second $ \zz $ by $ (a, b)\leftrightarrow (b^{-1}, a^{-1}) $;
		\item $ N(R_2)=G_1 \times G_2 $.
		\end{enumerate}
\end{prop}
\begin{proof}
The result follows from the group structure of $G=(\CC^*)^2 \rtimes D_8$.
\end{proof}

\subsection{Main error terms} This subsection is devoted to computing the main error terms for all the three stages of the partial desingularization.

\begin{prop} \label{prop:main0}For the group $ R_0 \cong (\CC^*)^2 $, the main term is 
	\begin{align*}
	P_t^{N(R_0)}(Z_{R_0}^{ss})(t^2+...+t^{2(\rk \calN^{R_0}-1)})&=\frac{t^2+...+t^{22}}{(1-t^4)(1-t^8)}\\
	&\equiv t^2+t^4+2t^6+2t^8+4t^{10} \ \mod t^{11}.\\
	\end{align*} 
\end{prop}

\begin{proof}
We saw in Proposition \ref{prop:blow-up} that $ Z_{R_0}^{ss} $ consists of a single point, and in Proposition \ref{prop:normal} the normaliser $ N(R_0)=G $, therefore:
$$ H^*_{N(R_0)}(Z_{R_0}^{ss})=H^*(B N(R_0))=H^*(B G). $$
We already showed in (\ref{eq:PGX}) that $H^*(B G)=\QQ[c_1^2c_2^2, \ c_1^2+c_2^2]$ with $\deg c_1=\deg c_2=2$, so
$$ P_t^{N(R_0)}(Z_{R_0}^{ss})=(1-t^4)^{-1}(1-t^8)^{-1}.$$
In (\ref{formula:rank}) we recalled how to compute the rank of the normal bundle:
$$ \rk \calN^{R_0}= \dim X-(\dim G+\dim Z_{R_0}^{ss} -\dim N(R_0))=12-(2+0-2)=12. $$
\end{proof}
	
\begin{prop} \label{prop:main1} For the group $ R_1 \cong \CC^* $, the main term is 
	\begin{align*}
	P_t^{N(R_1)}(Z_{R_1,1}^{ss})(t^2+...+t^{2(\rk \calN^{R_1}-1)})&=\frac{1+t^2+t^4}{1-t^2}(t^2+...+t^{14})\\
	&\equiv t^2+3t^4+6t^6+9t^8+12t^{10} \ \mod t^{11}.\\
	\end{align*} 
\end{prop}

\begin{proof} For brevity, write $ R=R_1 $ and $ N=N(R_1)=(\CC^*)^2\rtimes (\zz \times \zz) $ (see Proposition \ref{prop:blow-up} and Proposition \ref{prop:normal}). Recall that $ Z_{R, 1}^{ss} $ is the strict transform of $ Z_{R}^{ss} $ in $ X_1^{ss} $ under the first blow-up. 

We saw in Proposition \ref{prop:blow-up} that $Z_R\cong \PP^4$ and
\begin{equation} \label{eq:ZR1}
    \begin{aligned}
    Z_{R}^{ss}= & \ \PP\{ Ax_0^4y_1^4+Bx_0^3x_1y_0y_1^3+C x_0^2x_1^2y_0^2y_1^2+D x_0x_1^3y_0^3y_1+Ex_1^4y_0^4 \} \\
     & \smallsetminus \{ A=B=C=0, \ C=D=E=0 \}.
\end{aligned}
\end{equation}
In this system of coordinates, the centre of the first blow-up consists of the point $p=(0:0:1:0:0)$. Therefore
\begin{equation}\label{eq:ZR1blowup}
    Z_{R, 1}^{ss}=(\Bl_p Z_R^{ss})^{ss},
\end{equation}
because we recall that, after taking the proper transform, one should restrict only to the semistable points in $ X_2 \rightarrow X_1$ for the induced action of $ G $. We want to stress that the Kirwan blow-up is a blow-up, followed by a restriction to the semistable locus.

To compute $P_t^{N}(Z_{R,1}^{ss})$ we notice that we can use Theorem \ref{thm:cohkirblow}. Indeed, the restriction of the first blow-up to $Z_R^{ss}$ coincides with the unique step of Kirwan's procedure applied to the action of $N$ on $Z_R$. Hence by Theorem \ref{thm:cohkirblow}, we obtain: 
$$ P_t^N(Z_{R,1}^{ss})=P_t^N(Z_R^{ss})+P_t^N(\{p\})(t^2+t^4+t^6)-\sum_{0\neq \beta' \in \calB'}\frac{1}{w(\beta', R_0, N)}t^{2d(\PP\calN_p, \beta')}P_t^{N\cap \Stab \beta'}(Z_{\beta', R_0}^{ss}),$$
where $\calB'$ is the indexing set of the HKKN stratification induced on the exceptional divisor $\PP\calN_p\cong \PP^3$. We now clarify how to calculate all the contributions appearing in the equality above.

Firstly, we choose to compute $P_t^N(\{p\})$. The equivariant cohomology of a point is
$$ H^*_N(\{p\})=H^*(BN)=H^*(B(\CC^*)^2)^{\zz \times \zz}=\QQ[c_1, c_2]^{\zz \times \zz},$$
where $c_1$ and $c_2$ are the generating classes of the cohomology of $B(\CC^*)^2$ and have both degree $2$, while the action of $\zz \times \zz$ is described in Proposition \ref{prop:normal}(ii). By Molien's formula we obtain $P_t^N(\{p\})=(1-t^4)^{-2}$.

Secondly, we compute $P_t^N(Z_R^{ss})$. We can once again apply Theorem \ref{thm:equi} and Remark \ref{rmk:unstable}, namely we consider the HKKN equivatiantly perfect stratification induced by the action of $N$ on $Z_R$ and we find:
\begin{equation}\label{eq:main1}
    P_t^N(Z_R^{ss})=P_t^N(Z_R)-\sum_{0 \neq \beta \in \calB}t^{2\codim(S_{\beta})}P_t^{\Stab \beta}(Z_{\beta}^{ss}).
\end{equation}
The indexing set of the previous stratification is $\calB=\{(0,0), (2, -2), (4, -4)\}$ and the data can be summarised as follows:
\begin{center}
        \begin{tabular}{c|c|c|c}
        $\calB\smallsetminus \{(0,0)\}$ & $Z_{\beta}^{ss} \subset Z_R$  & $\Stab \beta$ & $\codim(S_{\beta})$ \\
        \hline
        $(2, -2)$ & $(0:1:0:0:0)$ & $(\CC^*)^2 \rtimes \zz$& 3\rule{0pt}{2.5ex} \\[1ex]
        $(4, -4)$ & $(1:0:0:0:0)$ & $(\CC^*)^2 \rtimes \zz$& 4
    \end{tabular} 
\end{center}
The extension $(\CC^*)^2 \rtimes \zz$ is given by the involution $ (a, b)\leftrightarrow (b^{-1}, a^{-1}) $. Recalling that $P_t^N(Z_R)=P_t(\PP^4)P_t(BN)$, we obtain:
$$ P_t^N(Z_R^{ss})=\frac{1+t^2+t^4+t^6+t^8}{(1-t^4)^2}-\frac{t^6+t^8}{(1-t^2)(1-t^4)}.$$

Finally, we need to consider the contribution coming from the stratification of the exceptional divisor $\PP \calN_p$. The indexing set of this stratification is $\calB'=\{(0,0), \pm(2, -2), \pm(4, -4)\}$ and the data we need to compute are summarised as follows:
\begin{center}
        \begin{tabular}{c|c|c|c|c}
        $\calB\smallsetminus \{(0,0)\}$   & $w(\beta', R_0, N)$ & $N \cap \Stab \beta'$ & $d(\PP\calN_p, \beta')$& $P_t^{N\cap \Stab \beta'}(Z_{\beta', R_0}^{ss})$ \\[1ex]
        \hline
        $\pm (2, -2)$   & 2 & $(\CC^*)^2 \rtimes \zz$& 2& $(1-t^2)^{-1}(1-t^4)^{-1}$ \rule{0pt}{2.5ex} \\[1ex]
        
        $\pm(4, -4)$  & 2 & $(\CC^*)^2 \rtimes \zz$ & 3 & $(1-t^2)^{-1}(1-t^4)^{-1}$\\
    \end{tabular}
\end{center}
The extension $(\CC^*)^2 \rtimes \zz$ is given by the involution $ (a, b)\leftrightarrow (b^{-1}, a^{-1}) $. By Remark \ref{rmk:cohextra}, the equivariant Hilbert-Poincar\'e polynomial of each stratum is
$$ P_t^{N\cap \Stab \beta'}(Z_{\beta', R_0}^{ss})=P_t^{N\cap \Stab \beta'}(\{p\})P_t(\PP^0)=P_t(B((\CC^*)^2\rtimes \zz))=(1-t^2)^{-1}(1-t^4)^{-1}.$$

Combining the three steps of calculations above leads to:
$$P_t^N(Z_{R,1}^{ss})=\frac{1+t^2+t^4+t^6+t^8}{(1-t^4)^2}-\frac{t^6+t^8}{(1-t^2)(1-t^4)}+\frac{t^2+t^4+t^6}{(1-t^4)^2}-\frac{t^4+t^6}{(1-t^2)(1-t^4)}=\frac{1+t^2+t^4}{1-t^2}. $$

To complete the proof of the Proposition \ref{prop:main1} we need to compute the rank of the normal bundle:
	$$ \rk \calN^R= \dim X-(\dim G+\dim Z_{R}^{ss} -\dim N)=12-(2+4-2)=8. $$	
\end{proof}
\begin{prop} \label{prop:main2} For the group $ R_2\cong \CC^* $, the main term is 
	\begin{align*}
	P_t^{N(R_2)}(Z_{R_2,1}^{ss})(t^2+...+t^{2(\rk \calN^{R_2}-1)})&=\frac{1}{1-t^2}(t^2+...+t^{18})\\
	&\equiv t^2+2t^4+3t^6+4t^8+5t^{10} \ \mod t^{11}.\\
	\end{align*} 
\end{prop}

\begin{proof}
For brevity, write $ R=R_2 $ and $ N=N(R_2)=G_1 \times G_2 $ (see Proposition \ref{prop:blow-up} and Proposition \ref{prop:normal}). We recall that $ Z_{R, 1}^{ss} $ is the strict transform of $ Z_{R}^{ss} $ in $ X_2^{ss} $ under the second blow-up. 

Proposition \ref{prop:blow-up} describes $Z_R\cong \PP^2$ and
\begin{equation}\label{eq:ZR2}
    Z_{R}^{ss}=\PP\{ Ax_0^4y_0^2y_1^2+Bx_0^2x_1^2y_0^2y_1^2+Cx_1^4y_0^2y_1^2 \}\smallsetminus \{ (0:0:1), (1:0:0) \}.
\end{equation}

In this coordinate system, the centre of the first blow-up corresponds to the point $p=(0:1:0)$. Hence we have
\begin{equation}\label{eq:ZR2blowup}
    Z_{R, 1}^{ss}=(\Bl_p Z_R^{ss})^{ss},
\end{equation}
because, after considering the proper transform, one should restrict only to the semistable points in $ X_2 \rightarrow X_1$ for the induced action of $ G $. We recall that the Kirwan blow-up is a blow-up operation, followed by a restriction to the semistable locus.

In order to calculate $P_t^{N}(Z_{R,1}^{ss})$ we can apply Theorem \ref{thm:cohkirblow}. Indeed, the restriction of the second blow-up to $Z_R^{ss}$ coincides with the unique step of Kirwan's procedure for the action of $N$ on $Z_R$. Hence by Theorem \ref{thm:cohkirblow}, we obtain: 
$$ P_t^N(Z_{R,1}^{ss})=P_t^N(Z_R^{ss})+P_t^N(\{p\})t^2-\sum_{0\neq \beta' \in \calB'}\frac{1}{w(\beta', R_0, N)}t^{2d(\PP\calN_p, \beta')}P_t^{N\cap \Stab \beta'}(Z_{\beta', R_0}^{ss}),$$
where $\calB'$ is the indexing set of the HKKN stratification induced on the exceptional divisor $\PP\calN_p\cong \PP^1$. We now explain how to compute all the contributions appearing in the equality above.

Firstly, we choose to compute $P_t^N(\{p\})$. The equivariant cohomology of a point is
$$ H^*_N(\{p\})=H^*(BN)=H^*(BG_1)\otimes H^*(BG_2)=\QQ[c_1]^{\zz}\otimes \QQ[d_1]^{\zz},$$
where $c_1$ and $d_1$ are the generating classes of the cohomology of $B\CC^*$ and have both degree $2$. In both cases the action of $\zz$ interchanges the cohomology class with its opposite. By Molien's formula we get $P_t^N(\{p\})=(1-t^4)^{-2}$.

Secondly, we calculate $P_t^N(Z_R^{ss})$. We can once again apply Theorem \ref{thm:equi} and Remark \ref{rmk:unstable}, namely we consider the HKKN equivatiantly perfect stratification induced by the action of $N$ on $Z_R$ and we get:
\begin{equation}\label{eq:main2}
    P_t^N(Z_R^{ss})=P_t^N(Z_R)-\sum_{0 \neq \beta \in \calB}t^{2\codim(S_{\beta})}P_t^{\Stab \beta}(Z_{\beta}^{ss}).
\end{equation}
The indexing set of the HKKN stratification is $\calB=\{(0,0), (4, 0)\}$ and the contributions can be summarised as follows:
\begin{center}
        \begin{tabular}{c|c|c|c}
        $\calB\smallsetminus \{(0,0)\}$ & $Z_{\beta}^{ss} \subset Z_R$  & $\Stab \beta$ & $\codim(S_{\beta})$ \\
        \hline
        $(4,0)$ & $(1:0:0)$ & $\CC^* \times G_2$& 2 \rule{0pt}{2.5ex} \\[1ex]
    \end{tabular}
\end{center}
Recalling that $P_t^N(Z_R)=P_t(\PP^4)P_t(BN)$ and $P_t^{\CC^*\times G_2}(\PP^0)=P_t(B(\CC^*\times G_2))=(1-t^2)^{-1}(1-t^4)^{-1}$, we obtain:
$$P_t^N(Z_R^{ss})=\frac{1+t^2+t^4}{(1-t^4)^2}-\frac{t^4}{(1-t^2)(1-t^4)}=\frac{1}{(1-t^4)^2}. $$
Finally, we need to take into consideration the contribution coming from the stratification of the exceptional divisor $\PP \calN_p$. The indexing set of this HKKN stratification is $\calB'=\{(0,0), \pm(4, 0)\}$ and the data we need to calculate are summarised as follows:
\begin{center}
        \begin{tabular}{c|c|c|c|c}
        $\calB\smallsetminus \{(0,0)\}$   & $w(\beta', R_0, N)$ & $N \cap \Stab \beta'$ & $d(\PP\calN_p, \beta')$& $P_t^{N\cap \Stab \beta'}(Z_{\beta', R_0}^{ss})$ \\[1ex]
        \hline
        $\pm(4, 0)$  & 2 & $\CC^* \times G_2$& 1& $(1-t^2)^{-1}(1-t^4)^{-1}$ \rule{0pt}{2.5ex} \\[1ex]
    \end{tabular}
\end{center}
 By Remark \ref{rmk:cohextra}, the equivariant Hilbert-Poincar\'e polynomial of each stratum is
$$ P_t^{N\cap \Stab \beta'}(Z_{\beta', R_0}^{ss})=P_t^{N\cap \Stab \beta'}(\{p\})P_t(\PP^0)=P_t(B(\CC^* \times G_2))=(1-t^2)^{-1}(1-t^4)^{-1}.$$

Putting together the three steps of calculations above leads to:
$$P_t^N(Z_{R,1}^{ss})=\frac{1}{(1-t^4)^2}+\frac{t^2}{(1-t^4)^2}-\frac{t^2}{(1-t^2)(1-t^4)}=\frac{1}{1-t^4}. $$

In order to complete the proof of the Proposition \ref{prop:main2}, we need to compute the rank of the normal bundle:
	$$ \rk \calN^R= \dim X-(\dim G+\dim Z_R^{ss} -\dim N)=12-(2+2-2)=10. $$	
\end{proof}

\subsection{Extra terms} To complete the computation of the contributions $ A_R(t) $, we need to calculate the extra terms, as stated in Theorem \ref{thm:cohkirblow}. The crucial point is to analyse for each $ R\in \calR $ the representation $ \rho: R \rightarrow \Aut (\calN_x^R) $ on the normal slice to the orbit $ G\cdot Z_R^{ss} $ at a generic point $ x\in Z_R^{ss} $. Since here we are dealing only with the local geometry around $ x $, we can restrict to consider the normal slice to the orbit $ G^0\cdot Z_R^{ss} $, which is the connected component of $ G\cdot Z_R^{ss} $ at $ x $.

\begin{lemma}\label{lem:weights0}
	For $ R=R_0$, $ \dim \calN_x^{R_0}=12 $, the weights of the representation $ \rho $ of $ R_0 $ on $ \calN_x^{R_0} $ are described by the diagram in Figure \ref{fig:weights0}.
\end{lemma}
\begin{proof}
Each monomial in $H^0(\calO_{\PP^1\times \PP^1}(4,4))^{\iota}$ is an eigenspace for the action of $ R_0=G^0 $. Hence $ H^0(\calO_{\PP^1\times \PP^1}(4,4))^{\iota} $ decomposes as a direct sum of one-dimensional representations of $ R_0 $ with multiplicities one, as described by the Hilbert diagram in Figure \ref{fig:weights}. The tangent space to the orbit $G\cdot Z_{R_0}^{ss}$ at $x=Z_{R_0}^{ss}$ is zero-dimensional and the group $R_0$ acts on it with weight $0$. Therefore the weights of the representation on the normal slice $\calN_x^{R_0} $ are all the ones in $H^0(\calO_{\PP^1\times \PP^1}(4,4))^{\iota}$ except for the origin: they are pictured as black dots in the Hilbert diagram of Figure \ref{fig:weights0}. 
\end{proof}

        \begin{figure}[htpb] 
		\begin{tikzpicture}[scale=.7]
		\draw [->] (-4.5,0) -- (4.5,0);
		\draw [->] (0, -4.5) -- (0, 4.5);
		
		\draw (-4,-4) node[circle,fill, inner sep=2pt]{};
		\draw (-4,0) node[circle,fill, inner sep=2pt]{};
		\draw (-4,4) node[circle,fill, inner sep=2pt]{};
		\draw (-2,-2) node[circle,fill, inner sep=2pt]{};
		\draw (-2,2) node[circle,fill, inner sep=2pt]{};
		\draw (0,-4) node[circle,fill, inner sep=2pt]{};
		\draw (0,4) node[circle,fill, inner sep=2pt]{};
		\draw (2,-2) node[circle,fill, inner sep=2pt]{};
		\draw (2,2) node[circle,fill, inner sep=2pt]{};
		\draw (4,-4) node[circle,fill, inner sep=2pt]{};
		\draw (4,0) node[circle,fill, inner sep=2pt]{};
		\draw (4,4) node[circle,fill, inner sep=2pt]{};

		\draw (4,0)  node[circle,fill,red,inner sep=1pt]{};
		\draw (4,0)  node[circle,draw,red, inner sep=3pt]{};
		\draw (4,-4)  node[circle,fill,red,inner sep=1pt]{};
		\draw (4,-4)  node[circle,draw,red, inner sep=3pt]{};
		\draw (2,-2)  node[circle,fill,red,inner sep=1pt]{};
		\draw (2,-2)  node[circle,draw,red, inner sep=3pt]{};
		\draw (2,0)  node[circle,fill,red,inner sep=1pt]{};
		\draw (2,0)  node[circle,draw,red, inner sep=3pt]{};
		\draw (2,2)  node[circle,fill,red,inner sep=1pt]{};
		\draw (2,2)  node[circle,draw,red, inner sep=3pt]{};
		\draw (4,4)  node[circle,fill,red,inner sep=1pt]{};
		\draw (4,4)  node[circle,draw,red, inner sep=3pt]{};
		\draw (0,4)  node[circle,fill,red,inner sep=1pt]{};
		\draw (0,4)  node[circle,draw,red, inner sep=3pt]{};
		\draw (-2,0)  node[circle,fill,red,inner sep=1pt]{};
		\draw (-2,0)  node[circle,draw,red, inner sep=3pt]{};
		\draw (0,2)  node[circle,fill,red,inner sep=1pt]{};
		\draw (0,2)  node[circle,draw,red, inner sep=3pt]{};
		\draw (-2,2)  node[circle,fill,red,inner sep=1pt]{};
		\draw (-2,2)  node[circle,draw,red, inner sep=3pt]{};
		\draw (-4,0)  node[circle,fill,red,inner sep=1pt]{};
		\draw (-4,0)  node[circle,draw,red, inner sep=3pt]{};
		\draw (-4,4)  node[circle,fill,red,inner sep=1pt]{};
		\draw (-4,4)  node[circle,draw,red, inner sep=3pt]{};
		\draw (-2,-2)  node[circle,fill,red,inner sep=1pt]{};
		\draw (-2,-2)  node[circle,draw,red, inner sep=3pt]{};
		\draw (-4,-4)  node[circle,fill,red,inner sep=1pt]{};
		\draw (-4,-4)  node[circle,draw,red, inner sep=3pt]{};
		\draw (0,-2)  node[circle,fill,red,inner sep=1pt]{};
		\draw (0,-2)  node[circle,draw,red, inner sep=3pt]{};
		\draw (0,-4)  node[circle,fill,red,inner sep=1pt]{};
		\draw (0,-4)  node[circle,draw,red, inner sep=3pt]{};
		
		\draw (6/5,2/5)  node[circle,fill,green,inner sep=1pt]{};
		\draw (6/5,2/5)  node[circle,draw,green, inner sep=3pt]{};
		\draw (8/5,4/5)  node[circle,fill,green,inner sep=1pt]{};
		\draw (8/5,4/5)  node[circle,draw,green, inner sep=3pt]{};
		\draw (12/5,4/5)  node[circle,fill,green,inner sep=1pt]{};
		\draw (12/5,4/5)  node[circle,draw, green,inner sep=3pt]{};
		
		\draw (-6/5,-2/5)  node[circle,fill,green,inner sep=1pt]{};
		\draw (-6/5,-2/5)  node[circle,draw,green, inner sep=3pt]{};
		\draw (-8/5,-4/5)  node[circle,fill,green,inner sep=1pt]{};
		\draw (-8/5,-4/5)  node[circle,draw,green, inner sep=3pt]{};
		\draw (-12/5,-4/5)  node[circle,fill,green,inner sep=1pt]{};
		\draw (-12/5,-4/5)  node[circle,draw, green,inner sep=3pt]{};
		
		\draw (-6/5,2/5)  node[circle,fill,green,inner sep=1pt]{};
		\draw (-6/5,2/5)  node[circle,draw,green, inner sep=3pt]{};
		\draw (-8/5,4/5)  node[circle,fill,green,inner sep=1pt]{};
		\draw (-8/5,4/5)  node[circle,draw,green, inner sep=3pt]{};
		\draw (-12/5,4/5)  node[circle,fill,green,inner sep=1pt]{};
		\draw (-12/5,4/5)  node[circle,draw, green,inner sep=3pt]{};
		
		\draw (6/5,-2/5)  node[circle,fill,green,inner sep=1pt]{};
		\draw (6/5,-2/5)  node[circle,draw,green, inner sep=3pt]{};
		\draw (8/5,-4/5)  node[circle,fill,green,inner sep=1pt]{};
		\draw (8/5,-4/5)  node[circle,draw,green, inner sep=3pt]{};
		\draw (12/5,-4/5)  node[circle,fill,green,inner sep=1pt]{};
		\draw (12/5,-4/5)  node[circle,draw, green,inner sep=3pt]{};
		
		\draw (-2/5,6/5)  node[circle,fill,green,inner sep=1pt]{};
		\draw (-2/5,6/5)  node[circle,draw,green, inner sep=3pt]{};
		\draw (-4/5,8/5)  node[circle,fill,green,inner sep=1pt]{};
		\draw (-4/5,8/5)  node[circle,draw,green, inner sep=3pt]{};
		\draw (-4/5,12/5)  node[circle,fill,green,inner sep=1pt]{};
		\draw (-4/5,12/5)  node[circle,draw, green,inner sep=3pt]{};
		
		\draw (-2/5,-6/5)  node[circle,fill,green,inner sep=1pt]{};
		\draw (-2/5,-6/5)  node[circle,draw,green, inner sep=3pt]{};
		\draw (-4/5,-8/5)  node[circle,fill,green,inner sep=1pt]{};
		\draw (-4/5,-8/5)  node[circle,draw,green, inner sep=3pt]{};
		\draw (-4/5,-12/5)  node[circle,fill,green,inner sep=1pt]{};
		\draw (-4/5,-12/5)  node[circle,draw, green,inner sep=3pt]{};
		
		\draw (2/5,-6/5)  node[circle,fill,green,inner sep=1pt]{};
		\draw (2/5,-6/5)  node[circle,draw,green, inner sep=3pt]{};
		\draw (4/5,-8/5)  node[circle,fill,green,inner sep=1pt]{};
		\draw (4/5,-8/5)  node[circle,draw,green, inner sep=3pt]{};
		\draw (4/5,-12/5)  node[circle,fill,green,inner sep=1pt]{};
		\draw (4/5,-12/5)  node[circle,draw, green,inner sep=3pt]{};
		
		\draw (2/5,6/5)  node[circle,fill,green,inner sep=1pt]{};
		\draw (2/5,6/5)  node[circle,draw,green, inner sep=3pt]{};
		\draw (4/5,8/5)  node[circle,fill,green,inner sep=1pt]{};
		\draw (4/5,8/5)  node[circle,draw,green, inner sep=3pt]{};
		\draw (4/5,12/5)  node[circle,fill,green,inner sep=1pt]{};
		\draw (4/5,12/5)  node[circle,draw, green,inner sep=3pt]{};
		\end{tikzpicture}
		\caption{\textit{Hilbert diagram} of the weights in the exceptional divisor $\PP\calN^{R_{0}}_x$.}\label{fig:weights0}
	\end{figure}
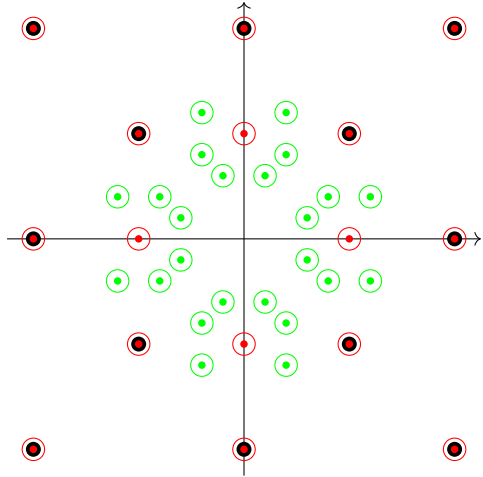

\begin{prop}\label{prop:extra0}
	For the group $ R_{0}\cong (\CC^*)^2 $ the extra term of $ A_{R_0}(t) $ is given by
	\begin{align*}
	\sum_{0\neq \beta'\in\calB(\rho)}\frac{1}{w(\beta', R_{0}, G)}t^{2d(\PP\calN^{R_{0}}, \beta')}P_t^{N(R_{0})\cap \Stab\beta'}(Z_{\beta', R_{0}}^{ss})&=\frac{t^{12}(1+2t^2+t^4-t^{12})}{(1-t^2)(1-t^4)}\\
	&\equiv 0 \ \mod t^{11}.
	\end{align*} 
\end{prop}
\begin{proof}
Let $ \{ \alpha_1,..., \alpha_{12} \} \subset  \mathrm{Lie}_{\mathbb{R}}(R_0) \cong \mathbb{R}^2  $ be the weights of the representation  $ \rho: R_0 \rightarrow \Aut(\calN_x^{R_0}) $ at $x=Z_{R_0}^{ss}$ as computed in Lemma \ref{lem:weights0}. After choosing a positive Weyl chamber for the adjoint action of $R_0\cong (\CC^*)^2$, which coincides with the whole $ \mathrm{Lie}_{\mathbb{R}}(R_0) $ in this case, we recall that an element $ \beta' \in \mathrm{Lie}_{\mathbb{R}}(R_0) $ belongs to the indexing set $ \calB(\rho) $ of the stratification if and only if $ \beta' $ is the closest point to the origin of the convex hull of some non-empty subset of $ \{ \alpha_1,..., \alpha_{12} \} $.

Table \ref{tab:R0} displays all the data required to compute the extra term for $R_0$. We notice that they clearly coincide with the information in Table \ref{tab:1} except for the codimension of the strata which is decreased by two, as the weight zero is missing. The value $ w(\beta', R_{0}, G) $ can be easily deduced from the diagram in Figure \ref{fig:weights0}, while the equivariant Hilbert-Poincar\'{e} series $P_t^{N(R_0) \cap \Stab \beta'}(Z_{\beta', \rho}^{ss})$ can be computed as in Lemma~\ref{lem:unstable1}, Lemma \ref{lem:unstable2} and Lemma \ref{lem:unstable3}.
\end{proof}

	\begin{table}[htpb]
		\centering
		\begin{tabular}{c|c|c|c|c}
			weights in $ \langle \beta \rangle $&$ w(\beta', R_{0}, G) $& $\Stab \beta $& $2d(\PP \calN ^{R_0},\beta')$&$P_t^{N(R_0) \cap \Stab \beta'}(Z_{\beta', \rho}^{ss})$\\[1ex]
			\hline \rule{0pt}{2.5ex}$ (4, -4) $&4& $ (\CC^*)^2\rtimes \zz $&22&$(1-t^2)^{-1}(1-t^4)^{-1} $\\[1ex]
			$ (4, 0), (2, -2), (0, -4) $&4& $ (\CC^*)^2\rtimes \zz $&16&$ (1+t^2-t^6)(1-t^2)^{-1}(1-t^4)^{-1} $\\[1ex]
			$ (4,4),(2,-2) $&8& $ (\CC^*)^2 $&16&$ (1-t^2)^{-1}$\\[1ex]
			$ (2,2),(0, -4) $&8& $ (\CC^*)^2 $&12&$ (1-t^2)^{-1}$\\[1ex]
			$ (4,4), (0, -4) $&8 &$ (\CC^*)^2 $&14&$ (1-t^2)^{-1} $\\[1ex]
			$ (2, 2), (2, -2) $&4 &$ \CC^*\times G_2 $&14&$ (1-t^2)^{-1} $\\[1ex]
			$ (4, 4), (4, 0), (4, -4) $&4 &$ \CC^*\times G_2 $&18&$ (1+t^2-t^6)(1-t^2)^{-1}(1-t^4)^{-1} $
		\end{tabular} 
		\caption{Cohomology of the unstable strata in the exceptional divisor.}\label{tab:R0}
	\end{table}

\begin{lemma}\label{lem:weights1}
	For $ R=R_1 $, $ \dim \calN_x^{R_1}=8 $, the weights of the representation $ \rho $ of $ R_1 $ on $ \calN_x^{R_1} $ are as follows with the respective multiplicities
	$$ (\pm 8)\times 1, \ (\pm 4)\times 3. $$
\end{lemma}

\begin{proof}
The torus $R_1$ acts on the coordinates $ ((x_0:x_1), (y_0:y_1)) $ of $\PP^1\times \PP^1$ diagonally. Thus each monomial in $H^0(\calO_{\PP^1\times \PP^1}(4,4))^{\iota}$ is an eigenspace for the action of $ R_1 $. Hence $ H^0(\calO_{\PP^1\times \PP^1}(4,4))^{\iota}=\CC^{13} $ decomposes as a sum of one-dimensional representations of $ R_1 $ with the following multiplicities of weights:
$$ (\pm 8) \times 1, \ (\pm 4) \times 3, \ (0) \times 5.$$
The orbit $G^0\cdot Z_{R_1}^{ss}$ is an open part of a linear subspace, since it coincides with $Z_{R_1}^{ss}$. Therefore the tangent space at every point $x\in G^0\cdot Z_{R_2}^{ss}$ can be identified, via the Euler sequence, with the corresponding vector subspace 
$$\langle x_0^4y_1^4, \ x_0^3x_1y_0y_1^3, \ x_0^2x_1^2y_0^2y_1^2, \ x_0x_1^3y_0^3y_1, \ x_1^4y_0^4 \rangle \subset H^0(\calO_{\PP^1\times \PP^1}(4,4))^{\iota}.$$ 
Each monomial spans an eigenspace for the action of $R_1$ with weight zero, because $R_1$ is contained in the stabiliser of every point $x\in G^0\cdot Z_{R_2}^{ss}$.

By subtracting the weights $(0)\times 5$ of the representation of the tangent space to the orbit from the weights of the representation of $ R_1 $ on $ H^0(\calO_{\PP^1\times \PP^1}(4,4))^{\iota} $, we obtain the weights of the action on the normal space.
\end{proof}
\begin{prop}\label{prop:extra1}
	For the group $ R_1\cong \CC^* $ the extra term of $ A_{R_1}(t) $ is given by
	\begin{align*}
	\sum_{0\neq \beta'\in\calB(\rho)}\frac{1}{w(\beta', R_1, G)}t^{2d(\PP\calN^{R_1}, \beta')}P_t^{N(R_1)\cap \Stab\beta'}(Z_{\beta', R_1}^{ss})&=\frac{1+2t^2+2t^4+t^6}{1-t^2}(t^8+t^{10}+t^{12}+t^{14})\\
	&\equiv t^8+4t^{10}  \ \mod t^{11}.
	\end{align*} 
\end{prop}
\begin{proof}
For brevity, we write $R=R_1$ and $N=N(R_1)$. By Lemma \ref{lem:weights1} we can take $ \calB(\rho)=\{\pm 8, \pm 4, 0 \} $ as indexing set of the stratification on the projective normal slice $\PP\calN_x^{R}$ at a point $x \in G\cdot Z_{R}^{ss}$. We can compute the codimension of the strata $ Z_{\beta', R}^{ss} $ via the formula (\ref{codim}):
	$$ d(\PP\calN_x^{R}, \beta')=n(\beta')-\dim(R/P_{\beta'}), $$
	where $ n(\beta') $ is the number of weights $ \alpha $ such that $ \alpha \cdot \beta'<||\beta'||^2 $ and $ P_{\beta'} $ is the associated parabolic subgroup. We have $ d(\pm 4)=4 $ and $ d(\pm 8)=7 $. Due to the symmetry, the coefficient for every weight is $ w(\beta', R, G)=2 $ and the stabiliser is $ \Stab \beta'=N\cap \Stab \beta'= (\CC^*)^2 \rtimes \zz $. The extension $(\CC^*)^2 \rtimes \zz$ is given by the involution $ (a, b)\leftrightarrow (b, a) $.
	
	By Remark \ref{rmk:cohextra}, we obtain for every $\beta' \in \calB (\rho) \setminus \{0\}$:
	$$ P_t^{N\cap \Stab \beta'}(Z_{\beta', R}^{ss})=P_t^{N\cap \Stab \beta'}(Z_{R,1}^{ss}) P_t(Z_{\beta', \rho}), $$
	because
	$$ Z_{\beta', \rho} = Z_{\beta', \rho}^{ss}=
	\begin{cases}
	\PP^2 & \beta'=\pm 4\\
	\PP^0 & \beta'=\pm 8.
	\end{cases}$$
    Therefore, we just need to compute $P_t^{(\CC^*)^2 \rtimes \zz}(Z_{R,1}^{ss})$ in a way similar to Proposition \ref{prop:main1}. Recall that by (\ref{eq:ZR1}) and (\ref{eq:ZR1blowup}) $Z_{R,1}^{ss}$ is isomorphic to the semistable locus in the blow-up of $Z_R^{ss}\subset Z_R^{ss}\cong \PP^4$ at $p=(0:0:1:0:0)$. By Theorem \ref{thm:cohkirblow}, the action of $(\CC^*)^2 \rtimes \zz$ on $Z_R$ leads to: 
    \begin{equation}\label{eq:extra1}
        \begin{aligned}
        P_t^{(\CC^*)^2 \rtimes \zz}(Z_{R,1}^{ss}) = & \ P_t^{(\CC^*)^2 \rtimes \zz}(Z_R^{ss})+P_t^{(\CC^*)^2 \rtimes \zz}(\{p\})(t^2+t^4+t^6)\\
        & -\sum_{0\neq \beta' \in \calB'}\frac{1}{w(\beta', R_0, (\CC^*)^2 \rtimes \zz)}t^{2d(\PP\calN_p, \beta')}P_t^{(\CC^*)^2 \rtimes \zz \cap \Stab \beta'}(Z_{\beta', R_0}^{ss}),
    \end{aligned}
    \end{equation}
where $\calB'$ is the indexing set of the HKKN stratification induced on the exceptional divisor $\PP\calN_p\cong \PP^3$. We now clarify how to calculate all the contributions appearing in the equality above.

Firstly, we choose to compute $P_t^{(\CC^*)^2 \rtimes \zz}(\{p\})$. The equivariant cohomology of a point is
$$ H^*_{(\CC^*)^2 \rtimes \zz}(\{p\})=H^*(B((\CC^*)^2 \rtimes \zz))=H^*(B\CC^*)^{\zz}=\QQ[c_1, c_2]^{\zz}=\QQ[c_1+c_2, c_1c_2],$$
where $c_1$ and $c_2$ are the generating classes of the cohomology of $B(\CC^*)^2$ and have both degree $2$. The action of $\zz$ interchanges the two classes, so we obtain $P_t^{(\CC^*)^2 \rtimes \zz}(\{p\})=(1-t^2)^{-1}(1-t^4)^{-1}$.

Secondly, we compute $P_t^{(\CC^*)^2 \rtimes \zz}(Z_R^{ss})$. We can once again apply Theorem \ref{thm:equi} and Remark \ref{rmk:unstable}, namely we consider the HKKN equivatiantly perfect stratification induced by the action of $(\CC^*)^2 \rtimes \zz$ on $Z_R$ and we find:
\begin{equation}\label{eq:extra1dopo}
P_t^{(\CC^*)^2 \rtimes \zz}(Z_R^{ss})=P_t^{(\CC^*)^2 \rtimes \zz}(Z_R)-\sum_{0 \neq \beta \in \calB}t^{2\codim(S_{\beta})}P_t^{\Stab \beta}(Z_{\beta}^{ss}).
\end{equation}
The indexing set of the previous stratification is $\calB=\{(0,0), (2, -2), (4, -4)\}$ and the data can be summarised as follows:
\begin{center}
    \begin{tabular}{c|c|c|c}
        $\calB\smallsetminus \{(0,0)\}$ & $Z_{\beta}^{ss} \subset Z_R$  & $\Stab \beta$ & $\codim(S_{\beta})$ \\
        \hline
        $(2, -2)$ & $(0:1:0:0:0)$ & $(\CC^*)^2$& 3\rule{0pt}{2.5ex} \\[1ex]
        $(4, -4)$ & $(1:0:0:0:0)$ & $(\CC^*)^2$& 4
    \end{tabular}
\end{center}
Recalling that $P_t^{(\CC^*)^2 \rtimes \zz}(Z_R)=P_t(\PP^4)P_t(B((\CC^*)^2 \rtimes \zz))$ and $P_t^{(\CC^*)^2}(\PP^0)=(1-t^2)^{-2}$, we obtain:
$$ P_t^{(\CC^*)^2 \rtimes \zz}(Z_R^{ss}) =\frac{1+t^2+t^4+t^6+t^8}{(1-t^2)(1-t^4)}-\frac{t^6+t^8}{(1-t^2)^2}=\frac{1+t^2+t^4-t^8-t^{10}}{(1-t^2)(1-t^4)}.$$

Finally, we need to consider the contribution coming from the stratification of the exceptional divisor $\PP \calN_p$. The indexing set of this stratification is $\calB'=\{(0,0), \pm(2, -2), \pm(4, -4)\}$ and the data we need to compute are summarised as follows:
\begin{center}
        \begin{tabular}{c|c|c|c|c}
        $\calB\smallsetminus \{(0,0)\}$   & $w(\beta', R_0, (\CC^*)^2 \rtimes \zz)$ & $(\CC^*)^2 \rtimes \zz \cap \Stab \beta'$ & $d(\PP\calN_p, \beta')$& $P_t^{(\CC^*)^2 \rtimes \zz \cap \Stab \beta'}(Z_{\beta', R_0}^{ss})$ \\[1ex]
        \hline
        $\pm (2, -2)$   & 2 & $(\CC^*)^2 $& 2& $(1-t^2)^{-2}$ \rule{0pt}{2.5ex} \\[1ex]
        
        $\pm(4, -4)$  & 2 & $(\CC^*)^2 $ & 3 & $(1-t^2)^{-2}$\\
    \end{tabular}
\end{center}
 By Remark \ref{rmk:cohextra}, the equivariant Hilbert-Poincar\'e polynomial of each stratum is
$$ P_t^{(\CC^*)^2 \rtimes \zz \cap \Stab \beta'}(Z_{\beta', R_0}^{ss})=P_t^{(\CC^*)^2 \rtimes \zz \cap \Stab \beta'}(\{p\})P_t(\PP^0)=P_t(B(\CC^*)^2)=(1-t^2)^{-2}.$$

Combining the three steps of calculations above leads to the result of (\ref{eq:extra2}):
$$P_t^{(\CC^*)^2 \rtimes \zz}(Z_{R,1}^{ss})=
\frac{1+t^2+t^4-t^8-t^{10}}{(1-t^2)(1-t^4)}+\frac{t^2+t^4+t^6}{(1-t^2)(1-t^4)}-\frac{t^4+t^6}{(1-t^2)^2}=\frac{1+2t^2+2t^4+t^6}{1-t^2}. $$
\end{proof}
\begin{lemma}\label{lem:weights2}
	For $ R=R_2 $, $ \dim \calN_x^{R_2}=10 $, the weights of the representation $ \rho $ of $ R_2 $ on $ \calN_x^{R_2} $ are as follows with the respective multiplicities
	$$ (\pm 4) \times 3, \ (\pm 2) \times 2. $$
\end{lemma}

\begin{proof}
The torus $R_2$ acts on the coordinates $ ((x_0:x_1), (y_0:y_1)) $ of $\PP^1\times \PP^1$ diagonally. Thus each monomial in $H^0(\calO_{\PP^1\times \PP^1}(4,4))^{\iota}$ is an eigenspace for the action of $ R_2 $. Hence $ H^0(\calO_{\PP^1\times \PP^1}(4,4))^{\iota}=\CC^{13} $ decomposes as a sum of one-dimensional representations of $ R_2 $ with the following multiplicities of weights:
$$ (\pm 4) \times 3, \ (\pm 2) \times 2, \ (0) \times 3.$$
The orbit $G^0\cdot Z_{R_2}^{ss}$ is an open part of a linear subspace, since it clearly coincides with $Z_{R_2}^{ss}$. Therefore the tangent space at every point $x\in G^0\cdot Z_{R_2}^{ss}$ can be identified, via the Euler sequence, with the corresponding vector subspace 
$$\langle x_0^4y_0^2y_1^2, \ x_0^2x_1^2y_0^2y_1^2, \ x_1^4y_0^2y_1^2 \rangle \subset H^0(\calO_{\PP^1\times \PP^1}(4,4))^{\iota}.$$ 
Each monomial spans an eigenspace for the action of $R_2$ with weight zero, because $R_2$ is contained in the stabiliser of every point $x\in G^0\cdot Z_{R_2}^{ss}$.

By subtracting the weights $(0)\times 3$ of the representation of the tangent space to the orbit from the weights of the representation of $ R_2 $ on $ H^0(\calO_{\PP^1\times \PP^1}(4,4))^{\iota} $, we obtain the weights of the action on the normal space.
\end{proof}
\begin{prop}\label{prop:extra2}
	For the group $ R_2\cong \CC^* $ the extra term of $ A_{R_2}(t) $ is given by
	\begin{align*}
	\sum_{0\neq \beta'\in\calB(\rho)}\frac{1}{w(\beta', R_2, G)}t^{2d(\PP\calN^{R_2}, \beta')}P_t^{N(R_2)\cap \Stab\beta'}(Z_{\beta', R_2}^{ss})&=\frac{1+t^2}{1-t^2}(t^{10}+t^{12}+t^{14}+t^{16}+t^{18})\\
	&\equiv t^{10} \ \mod t^{11}.
	\end{align*} 
\end{prop}

\begin{proof} For brevity, we write $R=R_2$ and $N=N(R_2)$. By Lemma \ref{lem:weights2} we can take $ \calB(\rho)=\{\pm 4, \pm 2, 0 \} $ as indexing set of the stratification on the projective normal slice $\PP\calN_x^{R}$ at a point $x \in G\cdot Z_{R}^{ss}$. We can compute the codimension of the strata $ Z_{\beta', R}^{ss} $ via the formula (\ref{codim}):
	$$ d(\PP\calN_x^{R}, \beta')=n(\beta')-\dim(R/P_{\beta'}), $$
	where $ n(\beta') $ is the number of weights $ \alpha $ such that $ \alpha \cdot \beta'<||\beta'||^2 $ and $ P_{\beta'} $ is the associated parabolic subgroup. We have $ d(\pm 2)=5 $ and $ d(\pm 4)=7 $. Due to the symmetry, the coefficient for every weight is $ w(\beta', R, G)=2 $ and the stabiliser is $ \Stab \beta'=N\cap \Stab \beta'=G_1\times \CC^* $.
	
	By Remark \ref{rmk:cohextra}, we obtain for every $\beta' \in \calB (\rho) \setminus \{0\}$:
	$$ P_t^{N\cap \Stab \beta'}(Z_{\beta', R}^{ss})=P_t^{N\cap \Stab \beta'}(Z_{R,1}^{ss}) P_t(Z_{\beta', \rho}), $$
	because
	$$ Z_{\beta', \rho} = Z_{\beta', \rho}^{ss}=
	\begin{cases}
	\PP^1 & \beta'=\pm 2\\
	\PP^2 & \beta'=\pm 4.
	\end{cases}$$
    Therefore, we just need to compute $P_t^{G_1\times \CC^*}(Z_{R,1}^{ss})$ in a way similar to Proposition \ref{prop:main2}. Recall that by (\ref{eq:ZR2}) and (\ref{eq:ZR2blowup}) $Z_{R,1}^{ss}$ is isomorphic to the semistable locus in the blow-up of $Z_R^{ss}\subset Z_R^{ss}\cong \PP^2$ at $p=(0:1:0)$. By Theorem \ref{thm:cohkirblow}, the action of $G_1\times \CC^*$ on $Z_R$ leads to: 
    \begin{equation}\label{eq:extra2}
        \begin{aligned}
        P_t^{G_1\times \CC^*}(Z_{R,1}^{ss}) = & \ P_t^{G_1\times \CC^*}(Z_R^{ss})+P_t^{G_1\times \CC^*}(\{p\})t^2\\
        & -\sum_{0\neq \beta' \in \calB'}\frac{1}{w(\beta', R_0, G_1\times \CC^*)}t^{2d(\PP\calN_p, \beta')}P_t^{G_1\times \CC^* \cap \Stab \beta'}(Z_{\beta', R_0}^{ss}),
    \end{aligned}
    \end{equation}
    
where $\calB'$ is the indexing set of the HKKN stratification induced on the exceptional divisor $\PP\calN_p\cong \PP^1$. We now clarify how to calculate all the contributions appearing in the equality above.

Firstly, we choose to compute $P_t^{G_1\times \CC^*}(\{p\})$. The equivariant cohomology of a point is
$$ H^*_{G_1\times \CC^*}(\{p\})=H^*(B(G_1\times \CC^*))=H^*(BG_1)\otimes H^*(B\CC^*)=\QQ[c_1]^{\zz}\otimes \QQ[d_1],$$
where $c_1$ and $d_1$ are the generating classes of the cohomology of $B\CC^*$ and have both degree $2$. The action of $\zz$ interchanges the cohomology class with its opposite. By Molien's formula we obtain $P_t^{G_1\times \CC^*}(\{p\})=(1-t^2)^{-1}(1-t^4)^{-1}$.

Secondly, we compute $P_t^{G_1\times \CC^*}(Z_R^{ss})$. We can once again apply Theorem \ref{thm:equi} and Remark \ref{rmk:unstable}, namely we consider the HKKN equivatiantly perfect stratification induced by the action of $G_1\times \CC^*$ on $Z_R$ and we find:
\begin{equation}\label{eq:extra2dopo}
    P_t^{G_1\times \CC^*}(Z_R^{ss})=P_t^{G_1\times \CC^*}(Z_R)-\sum_{0 \neq \beta \in \calB}t^{2\codim(S_{\beta})}P_t^{\Stab \beta}(Z_{\beta}^{ss}).
\end{equation}
The indexing set of the previous stratification is $\calB=\{(0,0), (4, 0)\}$ and the data can be summarised as follows:
\begin{center}
        \begin{tabular}{c|c|c|c}
        $\calB\smallsetminus \{(0,0)\}$ & $Z_{\beta}^{ss} \subset Z_R$  & $\Stab \beta$ & $\codim(S_{\beta})$ \\
        \hline
        $(4,0)$ & $(1:0:0)$ & $(\CC^*)^2$& 2 \rule{0pt}{2.5ex} \\[1ex]
    \end{tabular}
\end{center}
Recalling that $P_t^{G_1\times \CC^*}(Z_R)=P_t(\PP^4)P_t(B(G_1\times \CC^*))$ and $P_t^{(\CC^*)^2}(\PP^0)=(1-t^2)^{-2}$, we obtain:
$$ P_t^{G_1\times\CC^*}(Z_R^{ss}) =\frac{1+t^2+t^4}{(1-t^2)(1-t^4)}-\frac{t^4}{(1-t^2)^2}=\frac{1+t^2-t^6}{(1-t^2)(1-t^4)}.$$

Finally, we need to consider the contribution coming from the stratification of the exceptional divisor $\PP \calN_p$. The indexing set of this stratification is $\calB'=\{(0,0), \pm(4, 0)\}$ and the data we need to compute are summarised as follows:
\begin{center}
        \begin{tabular}{c|c|c|c|c}
        $\calB\smallsetminus \{(0,0)\}$   & $w(\beta', R_0, G_1\times \CC^*)$ & $G_1\times \CC^* \cap \Stab \beta'$ & $d(\PP\calN_p, \beta')$& $P_t^{G_1\times \CC^* \cap \Stab \beta'}(Z_{\beta', R_0}^{ss})$ \\[1ex]
        \hline
        $\pm(4, 0)$  & 2 & $(\CC^*)^2$& 1& $(1-t^2)^{-2}$ \rule{0pt}{2.5ex} \\[1ex]
    \end{tabular}
\end{center}
 By Remark \ref{rmk:cohextra}, the equivariant Hilbert-Poincar\'e polynomial of each stratum is
$$ P_t^{G_1\times \CC^* \cap \Stab \beta'}(Z_{\beta', R_0}^{ss})=P_t^{G_1\times \CC^* \cap \Stab \beta'}(\{p\})P_t(\PP^0)=P_t(B(\CC^*)^2)=(1-t^2)^{-2}.$$

Combining the three steps of calculations above leads to the result of (\ref{eq:extra2}):
$$P_t^{G_1\times \CC^*}(Z_{R,1}^{ss})=\frac{1+t^2-t^6}{(1-t^2)(1-t^4)}+\frac{t^2}{(1-t^2)(1-t^4)}-\frac{t^2}{(1-t^2)^2}=\frac{1+t^2}{1-t^2}. $$
\end{proof}

\subsection{Cohomology of $M^K$} We complete the proof of Theorem \ref{thm:cohoblow}.

\begin{proof}[Proof of Theorem \ref{thm:cohoblow}]
	From Theorem \ref{thm:cohkirblow}, we need to put all the previous results together to find the Betti numbers of the Kirwan partial desingularization $ M^K $. For the sake of readability, we report only the polynomials modulo $t^{10}$, but one can double-check the result with the entire Hilbert-Poincar\'{e} series and observe that Poincar\'{e} duality effectively holds.
\begin{align*}
P_t(M^K)&=P_t^G(\widetilde{X}^{ss})\equiv   \\
&1+t^2+2t^4+2t^6+4t^8+4t^{10} \tag{Semistable locus}\\
&+t^2+t^4+2t^6+2t^8+4t^{10}-0 \tag{Error term for  $ R_0 $} \\
&+t^2+3t^4+6t^6+9t^8+12t^{10}-(t^8+4t^{10})\tag{Error term for $ R_1 $}\\
&+t^2+2t^4+3t^6+4t^8+5t^{10}-t^{10} \tag{Error term for $ R_2 $}\\
&\equiv 1+4t^2+8t^4+13t^6+18t^8+20t^{10} \ \mod t^{11}.
\end{align*}
\end{proof}

\section{Intersection cohomology of the moduli space $ M^{GIT} $}
\label{sec:intersection}

In this Section, we compute the intersection cohomology of $M^{GIT}$ descending from $M^K$, and thus prove the following:

\begin{theorem}\label{thm:intM}
	The intersection Hilbert-Poincar\'{e} polynomial of $ M^{GIT} $ is
    $$ IP_t(M^{GIT})= 1+t^2+2t^4+2t^6+3t^8+3t^{10}+3t^{12}+2t^{14}+2t^{16}+t^{18}+t^{20}  $$
\end{theorem}

In the first part of the Section, we recall Kirwan's procedure to compare the cohomology of $ \widetilde{X}/\!\!/G $ and the intersection cohomology of $ X/\!\!/G $, as explained in \cite{Kir86}. This is in turn an application of the \textit{Decomposition Theorem} by Be\u{\i}linson, Bernstein, Deligne and Gabber (cf. \cite{BBD82}).

In the second part of the Section, instead of applying the Decomposition Theorem directly to the blow-down map $ M^K\rightarrow M^{GIT} $ at the level of GIT quotients, we follow Kirwan's results (see \cite{Kir86}) and study the variation of the intersection Betti numbers at the level of the parameter spaces $ X^{ss} $ and $ \widetilde{X}^{ss}$, under each stage of the resolution.

\subsection{General setting}
We start with the general setting, as in Section \S\ref{subsec:blow} and \S\ref{subsec:settingblowup}, of a projective manifold $ X $ acted on by a reductive group $ G $ together with a $G$-linearization. We suppose that we have already performed all the stages of the modification $ \widetilde{X}^{ss}\rightarrow X^{ss} $, indexed by the set $ \calR $, so that the \textit{Kirwan blow-up} $ \widetilde{X}/\!\!/G\rightarrow X/\!\!/G$ has been constructed by blowing up successively the (proper transforms of the) subvarieties $ Z_R^{ss}/\!\!/N(R) $.  Since the partial desingularization $ \widetilde{X}/\!\!/G $ has only finite quotient singularities, its intersection cohomology $ IH^*(\widetilde{X}/\!\!/G) $ with rational coefficients is isomorphic to the corresponding rational cohomology $ H^*(\widetilde{X}/\!\!/G) $, and so by the above results we know the Betti numbers of its intersection cohomology. Eventually, we will be able to find the intersection Betti numbers of $ X/\!\!/G $, by means of the following:

\begin{theorem}\cite[3.1]{Kir86}\label{thm:blowdown}
	In the above setting, the intersection Hilbert-Poincar\'{e} polynomial of the GIT quotient $ X/\!\!/G $ is related to that of the Kirwan blow-up via the equality
	$$ IP_t(X /\!\!/ G)=P_t(\widetilde{X}/\!\!/G)-\sum_{R\in\calR} B_R(t),$$
	where the error term is given by:
	$$ B_R(t)=\sum_{p+q=i}t^i \dim[H^p(\hat{Z}_{R}/\!\!/N^0(R))\otimes IH^{\hat{q}_R}(\PP \calN_x^R /\!\!/R)]^{\pi_0 N(R)}, $$
	where the integer $ \hat{q}_R=q-2 $ for $ q\leq \dim \PP\calN_x^R/\!\!/R $ and $ \hat{q}_R=q $ otherwise. The subvariety $ \hat{Z}_{R} $ is the strict transform of $ Z_R^{ss} $ in the appropriate stage of the resolution, while $ N(R)\subset G $ denotes the normaliser of $ R $. The GIT quotient $ \PP \calN_x^R /\!\!/R $ is constructed from the induced action of $ R $ on the normal slice $ \calN_x^R $ to the orbit $ G\cdot Z_R^{ss} $ in $ X^{ss} $ at a general point $ x\in Z_R^{ss} $.
\end{theorem}

\begin{remark} \label{rmk:Bt}
	If $ \hat{Z}_{R}/\!\!/N^0(R) $ is simply connected, which is always the case in our situation, then the action of $ \pi_0 N(R) $ on the tensor product splits \cite[\S 2]{Kir86}, thus the error term for the subgroup $ R $ is
	$$ B_R(t)=\sum_{p+q=i}t^i\dim H^p(\hat{Z}_{R}/\!\!/N(R))\cdot \dim IH^{\hat{q}_R}(\PP \calN_x^R/\!\!/R)^{\pi_0 N(R)}.$$
\end{remark}

\subsection{Cohomology of blow-downs for degree 2 Enriques suefaces} We want to apply Theorem \ref{thm:blowdown} to compute the intersection Betti numbers of the moduli space $M^{GIT}$ of degree 2 non-special Enriques surfaces. We now follow backwards the steps of the blow-down operations.

\begin{prop}\label{prop:down2}
	For the group $ R_2\cong \CC^* $, we have 
	\begin{enumerate}[(i)]
		\item $ Z_{R_2,1}/\!\!/N(R_2) $ is isomorphic to $ \PP^1 $;
		\item $ IP_t(\PP \calN_x^{R_2}/\!\!/R_2)= 1+2t^2+3t^4+4t^6+5t^8+4t^{10}+3t^{12}+2t^{14}+t^{16}$.
	\end{enumerate}
	The term $ B_{R_2}(t) $ is equal to
	$$ B_{R_2}(t)=t^2+2t^4+3t^6+4t^8+4t^{10}+4t^{12}+3t^{14}+2t^{16}+t^{18}.$$
\end{prop}

\begin{proof} The GIT quotient $  Z_{R_2,1}/\!\!/N(R_2)$ is a normal unirational curve, hence isomorphic to the projective line.

In Lemma \ref{lem:weights2} the weights of the representation $\rho: R_2\rightarrow \Aut(\calN_x^{R_2}) $ were computed. Since there are no strictly-semistable points, the GIT quotient $ \PP \calN_x^{R_2}/\!\!/R_2=\PP^9/\!\!/R_2 $ is a projective variety of dimension 8 with at worst finite quotient singularities. Therefore $ IP_t(\PP^9/\!\!/R_2)=P_t(\PP^9/\!\!/R_2)=P_t^{R_2}((\PP^9)^{ss}) $ and using the usual $ R_2 $-equivariantly perfect stratification (see Theorems \ref{thm:strata} and \ref{thm:equi}) we obtain: 
	\begin{align*}
	P_t^{R_2}((\PP^9)^{ss})&=P_t(\PP^9)P_t(BR_2)-\sum_{0\neq \beta'\in \calB(\rho)}t^{2d(\beta')}P_t^{R_2}(S_{\beta'})\\
	 &=\frac{1+...+t^{18}}{1-t^2}-2\frac{t^{10}(1+t^2)+t^{14}(1+t^2+t^4)}{1-t^2}.
	\end{align*}
	Now we need to know the dimension of $ IH^{\hat{q}}(\PP^9/\!\!/R_2)^{\pi_0 N(R_2)} $. The action of $ \pi_0 N(R_2)\cong \zz \times \zz$ on the cohomology of $ \PP^9 $ is trivial, while its action on $ R_2 $ is as follows: the first factor acts trivially and the second one acts by inversion. Moreover, $ \pi_0 N(R_2) $ acts on the strata interchanging the positive-indexed ones with the negative-indexed ones:
	\begin{equation}\label{eq:down2}
	    \begin{aligned}
	    IP_t(\PP^9/\!\!/R_2)^{^{\pi_0 N(R_2)}} &=\frac{1+...+t^{18}}{1-t^4}-\frac{t^{10}+...+t^{18}}{1-t^2} \\
	    &= 1+t^2+2t^4+2t^6+3t^8+2t^{10}+2t^{12}+t^{14}+t^{16}.
	    \end{aligned}
	\end{equation}
	Now the final statement easily follows from the definition of $ B_{R_2}(t) $ in Theorem \ref{thm:blowdown}.
\end{proof}
\begin{prop} \label{prop:down1}
	For the group $ R_1\cong \CC^* $, we have 
	\begin{enumerate}[(i)]
		\item $ Z_{R_1,1}/\!\!/N(R_1) $ is a simply connected threefold and $ P_t(Z_{R_1,1}/\!\!/N(R_1))=1+2t^2+2t^4+t^6 $;
		\item $ IP_t(\PP \calN_x^{R_1}/\!\!/R_1)=1+2t^2+3t^4+4t^6+3t^8+2t^{10}+t^{12}. $
	\end{enumerate}
	The term $ B_{R_1}(t) $ is equal to
	$$ B_{R_1}(t)=t^2+3t^4+6t^6+9t^8+10t^{10}+9t^{12}+6t^{14}+3t^{16}+t^{18}$$
\end{prop}

\begin{proof} For brevity we write $ R=R_1 $, $ N=N(R_1) $ and $ \PP^7\cong \PP \calN_x^{R_1} $. The GIT quotient $  Z_{R,1}/\!\!/N$ is a unirational threefold with finite quotient singularities, hence simply connected by \cite[Theorem 7.8.1]{Kol93}. Its cohomology can be computed by means of the equality \cite[1.17]{Kir86}:
	$$ H_{N}^*(Z_{R,1}^{ss})=(H^*(Z_{R,1}/\!\!/N^0)\otimes H^*(BR))^{\pi_0 N}. $$
	The action of $ \pi_0 N $ splits on the tensor product, because also $ Z_{R,1}/\!\!/N^0 $ is simply connected, giving:
	$$ H^*_N(Z_{R,1}^{ss})=H^*(Z_{R,1}/\!\!/N)\otimes H^*(BR)^{\pi_0 N}.$$
	Recall that $ \pi_0 N=\zz \times \zz $: the first factor acts on $ R\cong \CC^* $ trivially, while the second one acts by inversion. Therefore
	$$ H^*(BR)^{\pi_0 N}=\QQ[c]^{\zz \times \zz}=\QQ[c^2], \ \deg(c)=2. $$
	In the proof of Proposition \ref{prop:main1}, we have already computed $ P_t^N(Z_{R,1}^{ss}) $, thus
	$$ P_t(Z_{R,1}/\!\!/N)=\frac{1+t^2+t^4}{1-t^2}(1-t^4)=1+2t^2+2t^4+t^6, $$
	completing the proof of (i). 
	
	In Lemma \ref{lem:weights1} the weights of the representation $\rho: R\rightarrow \Aut(\calN_x^R) $ were computed. Since there are no strictly-semistable points, the GIT quotient $ \PP^7/\!\!/R $ is a projective variety of dimension 6 with at worst finite quotient singularities. Therefore $ IP_t(\PP^7/\!\!/R)=P_t(\PP^7/\!\!/R)=P_t^{R}((\PP^7)^{ss}) $ and using the usual $ R $-equivariantly perfect stratification (see Theorems \ref{thm:strata} and \ref{thm:equi}) we obtain: 
	\begin{align*}
	P_t^{R}((\PP^7)^{ss})&=P_t(\PP^7)P_t(BR)-\sum_{0\neq \beta'\in \calB(\rho)}t^{2d(\beta')}P_t^R(S_{\beta'})\\
	&=\frac{1+...+t^{14}}{1-t^2}-2\frac{t^8(1+t^2+t^4)+t^{14}}{1-t^2}\\
	&=1+2t^2+3t^4+4t^6+3t^8+2t^{10}+t^{12}.
	\end{align*}
	Now we need to know the dimensions $ \dim IH^{\hat{q}}(\PP^7/\!\!/R)^{\pi_0 N} $. The action of $ \pi_0 N\cong \zz \times \zz$ on the cohomology of $ \PP^7 $ is trivial, while its action on $ R $ was explained above. Moreover, $ \pi_0 N $ acts on the strata interchanging the positive-indexed ones with the negative-indexed ones:
	\begin{equation}\label{eq:down1}
	    \begin{aligned}
	    IP_t(\PP^7/\!\!/R)^{^{\pi_0 N}} &=\frac{1+...+t^{14}}{1-t^4}-\frac{t^8+...+t^{14}}{1-t^2} \\
	    &= 1+t^2+2t^4+2t^6+2t^8+t^{10}+t^{12}.
	    \end{aligned}
	\end{equation}
	Now the final statement easily follows from the definition of $ B_R(t) $ in Theorem \ref{thm:blowdown}.
\end{proof}

\begin{prop}
	For the group $ R_0\cong (\CC^*)^2 $, we have 
	\begin{align*}
	B_{R_0}(t)&=\sum_{2 \le q \le 18}t^q \dim IH^{\hat{q}_{R_0}}(\PP \calN_x^{R_0}/\!\!/R_0)^{\pi_0 N(R_0)} \\
	&=t^2+t^4+2t^6+2t^8+3t^{10}+2t^{12}+2t^{14}+t^{16}+t^{18}.
	\end{align*}
\end{prop}

\begin{proof}
For brevity we write $ R=R_0 $, $ N=N(R_0) $ and $ \PP^{11}\cong \PP \calN_x^{R_0} $. Clearly $ Z_R/\!\!/N $ is a point, thus we have to compute only the invariant intersection cohomology of the GIT quotient $ \PP^{11}/\!\!/R $. By looking at the weights of the representation of $R$ on $\PP^{11}$ from Lemma \ref{lem:weights0}, we find that this action unfortunately gives rise to strictly polystable points, hence we need to perform all the Kirwan procedure again in this case. We also need to take care of the invariants with respect to the action of the finite group $ \pi_0 N\cong D_8 $ at every step.

The first step is considering the $ R $-equivarintly perfect stratification of $ \PP^{11} $, as explained in Theorems \ref{thm:strata} and \ref{thm:equi}. This stratification was already considered in Proposition \ref{prop:extra0}, leading to
\begin{equation}
    P_t^R((\PP^{11})^{ss})^{\pi_0 N}=\frac{1+...+t^{22}}{(1-t^4)(1-t^8)}-\frac{t^{12}(1+2t^2+t^4-t^{12})}{(1-t^2)(1-t^4)}.
\end{equation}

The first term in the above expression comes from the $R$-equivariant cohomology of $\PP^{11}$ and can be computed as in (\ref{eq:PGX}), while the second one is the sum of the contributions from the unstable strata from Proposition \ref{prop:extra0}. The group $\pi_0N$ acts trivially on the $R$-equivariant cohomology of $\PP^{11}$ as in \ref{eq:HGX}, while it identifies the unstable strata in the same orbit under the action of the Weyl group of $G$ (cf. Lemma \ref{lem:weights0}).  

The second step of Kirwan's method amounts to blowing up the strictly semistable loci in $(\PP^{11})^{ss}$, which are indexed by $\calR^0=\{R_1, R_2, R_3, R_4\}$, where $R_1$ and $R_2$ are defined as in Proposition \ref{prop:blow-up}, while 
$$ R_3=\{ (t,t^{-1})\in R: t\in \CC^* \}\cong \CC^* \text{ and } R_4=\{ (t,1)\in R: t\in \CC^* \}\cong \CC^*.$$ 
The fixed loci of these subgroups are permuted by the action of $\pi_0 N$. Indeed, $Z_{R_1}^0$ is isomorphic to $Z_{R_3}^0$ and they are interchanged by the reflection $\langle \sigma \rangle < \pi_0 N$ along the x-axis (cf. Figure \ref{fig:weights0}). Moreover, $Z_{R_2}^0$ is isomorphic to $Z_{R_4}^0$ and they are interchanged by the reflection $\langle \tau \rangle < \pi_0 N$ along the diagonal (cf. Figure \ref{fig:weights0}). In the following we give the description of the fixed loci $Z_{R_i}^0$ for $i=1, \ldots, 4$ and the weights of the action of $R$ from Lemma \ref{lem:weights0}:
\begin{enumerate}[(i)]
    \item $Z_{R_1}^0\cong \PP^3$ and $(Z_{R_1}^0)^{ss}=\PP^3 \smallsetminus \{z_0=z_1=0, \ z_2=z_3=0\}$ because
    $$(a, b)\cdot (z_0:z_1:z_2:z_3)=(a^{-4}b^{4}z_0:a^{-2}b^{2}z_1:a^{2}b^{-2}z_2:a^{4}b^{-4}z_3), \quad (a, b)\in R, \ z\in Z_{R_1}^0.$$ The same holds for $Z_{R_3}^0$.
    \item $Z_{R_2}^0\cong \PP^1$ and $(Z_{R_2}^0)^{ss}=\PP^1\smallsetminus \{(0:1), \ (1:0)\}$ because
    $$ (a, b)\cdot (z_0: z_1)=(a^{-4}z_0: a^{4}z_1), \quad (a, b)\in R, \ z\in Z_{R_2}^0.$$
    The same holds for $Z_{R_4}^0$.
\end{enumerate}

To construct the Kirwan blow-up $\widetilde{\PP^{11}}/\!\!/R$, we need to blow up the orbit loci $ \bigcup_{i=1}^4 R\cdot (Z^0_{R_i})^{ss} $. Notice that the order of the resolutions is irrelevant, since the centres of the blow-ups are disjoint. 
$$ \begin{tikzcd}
(\widetilde{\PP^{11}})^{ss}=(\Bl_{\bigcup_{i=1}^4 R\cdot (Z^0_{R_i})^{ss}}(\PP^{11})^{ss})^{ss} \arrow{r} \arrow{d} & (\PP^{11})^{ss} \arrow{d} \\
\widetilde{\PP^{11}}/\!\!/R \arrow{r}&  \PP^{11}/\!\!/R
\end{tikzcd} $$

Following Theorem \ref{thm:cohkirblow}, we can now compute the cohomology of $\widetilde{\PP^{11}}/\!\!/R$.

\begin{claim*}
The $\pi_0 N$-equivariant cohomology of the Kirwan blow-up $\widetilde{\PP^{11}}/\!\!/R$ is:
$$ P_t(\widetilde{\PP^{11}}/\!\!/R)^{\pi_0 N}=1+3t^2+5t^4+8t^6+10t^8+10t^{10}+8t^{12}+5t^{14}+3t^{16}+t^{18}. $$
\end{claim*}

\begin{proof}[Proof of Claim]
By Theorem \ref{thm:cohkirblow}, we have
\begin{equation}\label{eq:invAt}
    \begin{aligned}
    P_t(\widetilde{\PP^{11}}/\!\!/R)^{\pi_0 N}=P_t^R((\widetilde{\PP^{11}})^{ss})^{\pi_0 N}&=P_t^R((\PP^{11})^{ss})^{\pi_0N}+\left( \sum_{i=1}^4 A^0_{R_i}(t) \right)^{\pi_0 N}\\
    & =P_t^R((\PP^{11})^{ss})^{\pi_0N}+(A^0_{R_1}(t))^{\pi_0 N/\langle \sigma \rangle}+(A^0_{R_2}(t))^{\pi_0 N/\langle \tau \rangle}.
    \end{aligned}
\end{equation}
The last equality follows from the fact that the fixed loci and consequently the exceptional divisors are permuted by $ \pi_0 N$, as explained above. Hence we need to calculate the two contributions $ (A^0_{R_1}(t))^{\pi_0 N/\langle \sigma \rangle}$ and $ (A^0_{R_2}(t))^{\pi_0 N/\langle \tau \rangle}$ coming from the blow-ups. We distinguish the two cases.
\begin{enumerate}[(i)]
    \item The main term of $ (A^0_{R_1}(t))^{\pi_0 N/\langle \sigma \rangle}$ is
    $$ \left( \frac{1+t^2+t^4+t^6}{(1-t^4)^2}-\frac{t^4+t^6}{(1-t^2)(1-t^4)} \right)(t^2+\ldots + t^{14}),$$
    where $P^R_t((Z_{R_1}^0)^{ss})^{\pi_0 N/\langle \sigma \rangle}=P_t^{(\CC^*)^2\rtimes (\zz \times \zz)}((Z^0_{R_1})^{ss}) $ has been computed using Theorem \ref{thm:equi} and it is completely analogous to the calculation of (\ref{eq:main1}) in Proposition \ref{prop:main1}, while $\rk \calN^{R_1}=8$ in this case. The extra term of $ (A^0_{R_1}(t))^{\pi_0 N/\langle \sigma \rangle}$ is
    $$ \left( \frac{1+t^2+t^4+t^6}{(1-t^2)(1-t^4)}-\frac{t^4+t^6}{(1-t^2)^2} \right)(t^{8}+\ldots + t^{14}), $$
    where $P^R_t((Z_{R_1}^0)^{ss})^{\pi_0 N/\langle \sigma \rangle}=P_t^{(\CC^*)^2\rtimes \zz}((Z^0_{R_1})^{ss}) $ has been computed using Theorem \ref{thm:equi} and it is totally similar to the calculation of (\ref{eq:extra1dopo}) in Proposition \ref{prop:extra1}.
    
    \item The main term of $ (A^0_{R_2}(t))^{\pi_0 N/\langle \tau \rangle}$ is
    $$ \left( \frac{1+t^2}{(1-t^4)^2}-\frac{t^2}{(1-t^2)(1-t^4)} \right)(t^2+\ldots + t^{18}),$$
    where $P^R_t((Z_{R_2}^0)^{ss})^{\pi_0 N/\langle \tau \rangle}=P_t^{G_1\times G_2}((Z^0_{R_2})^{ss}) $ has been computed using Theorem \ref{thm:equi} and it is completely analogous to the calculation of (\ref{eq:main2}) in Proposition \ref{prop:main2}, while $\rk \calN^{R_2}=10$ in this case. The extra term of $ (A^0_{R_2}(t))^{\pi_0 N/\langle \tau \rangle}$ is
    $$ \left( \frac{1+t^2}{(1-t^2)(1-t^4)}-\frac{t^2}{(1-t^2)^2} \right)(t^{10}+\ldots + t^{18}), $$
    where $P^R_t((Z_{R_2}^0)^{ss})^{\pi_0 N/\langle \tau \rangle}=P_t^{G_1\times \CC^*}((Z^0_{R_2})^{ss}) $ has been computed using Theorem \ref{thm:equi} and it is totally similar to the calculation of (\ref{eq:extra2dopo}) in Proposition \ref{prop:extra2}.
\end{enumerate}
By summing and subtracting appropriately the previous terms according to Theorem \ref{thm:cohkirblow}, the result follows.
\end{proof}

The third step of Kirwan's procedure consists of computing the intersection cohomology of $\PP^{11}/\!\!/R $ descending from $\widetilde{\PP^{11}}/\!\!/R $ following Theorem \ref{thm:blowdown}. Since we need only the invariant part of $IP_t(\PP^{11}/\!\!/R)$ under $\pi_0 N$, we argue as in (\ref{eq:invAt}) and find:
\begin{equation*}\label{eq:Bt0}
    \begin{aligned}
    IP_t(\PP^{11}/\!\!/R)^{\pi_0 N}&  = P_t(\widetilde{\PP^{11}}/\!\!/R)^{\pi_0 N}-\left( \sum_{i=1}^4 B^0_{R_i}(t) \right)^{\pi_0 N}\\
    & =P_t(\widetilde{\PP^{11}}/\!\!/R)^{\pi_0 N}-(B^0_{R_1}(t))^{\pi_0 N/\langle \sigma \rangle}-(B^0_{R_2}(t))^{\pi_0 N/\langle \tau \rangle},
    \end{aligned}
\end{equation*}
where $B^0_{R_1}(t)$ and $B^0_{R_2}(t)$ are defined in Theorem \ref{thm:blowdown}. We now calculate the invariant part of these two contributions:

\begin{enumerate}[(i)]
    \item $Z_{Z_1}^0/\!\!/R$ is a simply connected surface by \cite[Theorem 7.8.1]{Kol93}, because it is unirational and has only finite quotient singularities. Hence we can compute $(B^0_{R_1}(t))^{\pi_0 N/\langle \sigma \rangle}$ by using Remark \ref{rmk:Bt}. The cohomology $P_t(Z_{R_1}^0/\!\!/R)^{\pi_0 N/\langle \sigma \rangle}$ can be calculated by means of the equality \cite[1.17]{Kir86} in a totally analogous way to Proposition \ref{prop:down1} (i):
    $$ P_t(Z_{R_1}^0/\!\!/R)^{\pi_0 N/\langle \sigma \rangle}= \dfrac{P_t^R((Z_{R_1}^0)^{ss})^{\pi_0 N/\langle \sigma \rangle}}{P_t(B R_1)^{\pi_0 N/\langle \sigma \rangle}}=\frac{1+t^2+t^4}{1-t^4} (1-t^4)=1+t^2+t^4.$$
    Since the normal bundle of $R\cdot (Z^0_{R_1})^{ss}$ coincides with the one considered in Proposition \ref{prop:down1}, we obtain from (\ref{eq:down1}) that
    $$ IP_t(\PP \calN_x^{R_1}/\!\!/R_1)^{\pi_0 N/\langle \sigma \rangle}=1+t^2+2t^4+2t^6+2t^8+t^{10}+t^{12}. $$
    By Theorem \ref{thm:blowdown} and Remark \ref{rmk:Bt} we find:
    $$ (B^0_{R_1}(t))^{\pi_0 N/\langle \sigma \rangle}=t^2+2t^4+4t^6+5t^8+5t^{10}+4t^{12}+2t^{14}+t^{16}. $$
    \item $Z_{Z_1}^0/\!\!/R$ is a point. Since the normal bundle of $R\cdot (Z^0_{R_2})^{ss}$ coincides with the one considered in Proposition \ref{prop:down2}, we obtain from (\ref{eq:down2}) that
    $$ IP_t(\PP \calN_x^{R_2}/\!\!/R_2)^{\pi_0 N/\langle \sigma \rangle}=1+t^2+2t^4+2t^6+3t^8+2t^{10}+2t^{12}+t^{14}+t^{16}. $$
    By Theorem \ref{thm:blowdown} and Remark \ref{rmk:Bt} we find:
    $$ (B^0_{R_2}(t))^{\pi_0 N/\langle \tau \rangle}=t^2+t^4+2t^6+2t^8+2t^{10}+2t^{12}+t^{14}+t^{16}. $$
\end{enumerate}
By (\ref{eq:Bt0}), the blow-down operations give
$$ IP_t(\PP^{11}/\!\!/R)^{\pi_0 N}=1+t^2+2t^4+2t^6+3t^8+3t^{10}+2t^{12}+2t^{14}+t^{16}+t^{18}. $$
Now the result follows the definition of $B_{R_0}(t)$.
\end{proof}

\subsection{Intersection cohomology of $ M^{GIT} $} We complete the proof of Theorem \ref{thm:intM}.

\begin{proof}[Proof of Theorem \ref{thm:intM}]
From Theorem \ref{thm:blowdown} putting all the previous results together, we obtain that the intersection Hilbert-Poincar\'{e} polynomial of the moduli space of non-special degree 2 Enriques surfaces $ M^{GIT}=X/\!\!/G $ is 
\begin{align*}
IP_t(M^{GIT})&=P_t(M^K)-\sum_{R\in \calR} B_R(t)\\
&=P_t^G(X^{ss})+\sum_{R\in \calR}(A_R(t)-B_R(t))\\
&\equiv 1+t^2+2t^4+2t^6+4t^8+4t^{10}+(t^{10}-t^8-2t^{10}+0) \ \mod t^{10}\\
&\equiv 1+t^2+2t^4+2t^6+3t^8+3t^{10}  \ \mod t^{10}.
\end{align*}
\end{proof}

Together with Theorem \ref{thm:cohoblow}, this also completes the proof of the main Theorem \ref{thm:main}.

\begin{remark}\label{rmk:coho}(cf. \cite[3.4]{Kir86})
	As a by-product of our result, we are able to determine the ordinary Betti numbers
	$$ H^i(M^{GIT})=IH^i(M^{GIT}) \ \mathrm{for} \ 13 \leq i\leq 20  $$ 
	and
	$$ H^i(X^{s}/G)=IH^i(M^{GIT}) \ \mathrm{for} \ 0 \leq i\leq 7,  $$ 
	where $ X^{s}/G=M^{GIT} \smallsetminus \bigcup_{R\in \calR} Z_R/\!\!/N(R) $ is the orbit space of GIT-stable curves.
\end{remark}

	\bibliographystyle{alpha}
	\bibliography{References_Enriques}
	
\end{document}